\documentclass{article}
\usepackage[utf8]{inputenc}
\usepackage{amsmath,amsfonts,euscript,amscd,amsthm,amssymb,upref,graphics,xcolor,verbatim, float, placeins,mathrsfs, mathtools,hyperref}
\usepackage[all]{xy}
\usepackage[normalem]{ulem} 
\usepackage[shortlabels]{enumitem}

\swapnumbers

\theoremstyle{plain}

\newtheorem{theorem}[subsection]{Theorem}

\newtheorem*{thmnonumber}{Theorem}
\newtheorem{proposition}[subsection]{Proposition}
\newtheorem{lemma}[subsection]{Lemma}
\newtheorem{corollary}[subsection]{Corollary}

\newtheorem*{propnonumber}{Proposition}
\newtheorem*{lemma*}{Lemma}

\newtheorem*{theoremA}{Theorem A}

\theoremstyle{definition}
\newtheorem{definition}[subsection]{Definition}

\newtheorem{nothing}[subsection]{}
\newtheorem{nothing*}[subsection]{}
\newtheorem{example}[subsection]{Example}

\newtheorem{question}[subsection]{Question}

\newtheorem{notation}[subsection]{Notation}

\newtheorem{assumptions}[subsection]{Assumptions}
\newtheorem{assumption}[subsection]{Assumption}
\newtheorem{remark}[subsection]{Remark}

\newtheorem{subnothing*}[subtheorem]{}

\theoremstyle{remark}

{\begin{enumerate}
		
		}{\end{enumerate}}

\newcommand{\Spec}{		\operatorname{{\rm Spec}}}
\newcommand{\Proj}{		\operatorname{{\rm Proj}}}

\newcommand{\height}{		\operatorname{{\rm ht}}}

\newcommand{\trdeg}{	\operatorname{{\rm trdeg}}}
\newcommand{\Frac}{		\operatorname{{\rm Frac}}}

\newcommand{\Sing}{		\operatorname{{\rm Sing}}}

\renewcommand{\div}{	\operatorname{{\rm div}}}

\newcommand{\lcm}{		\operatorname{{\rm lcm}}}

\newcommand{\cotype}{		\operatorname{{\rm cotype}}}

\newcommand{\Div}{		\operatorname{{\rm Div}}}

\newcommand{\lb}{\langle}
\newcommand{\rb}{\rangle}

\newcommand{\codim}{\operatorname{{\rm\bf codim}}}

\newcommand{\setspec}[2]{\big\{\,#1\, \mid \,#2\, \big\}}

\newlength{\mylength}
\newcommand{\Integ}{\ensuremath{\mathbb{Z}}}
\newcommand{\Nat}{\ensuremath{\mathbb{N}}}
\newcommand{\Rat}{\ensuremath{\mathbb{Q}}}
\newcommand{\Comp}{\ensuremath{\mathbb{C}}}

\newcommand{\aff}{\ensuremath{\mathbb{A}}}
\newcommand{\proj}{\ensuremath{\mathbb{P}}}
\newcommand{\bk}{{\ensuremath{\rm \bf k}}}

\newcommand{\pgoth}{\mathfrak{p}}

\newcommand{\qgoth}{\mathfrak{q}}
\newcommand{\hgoth}{\mathfrak{h}}

\newcommand{\PPP}{\mathbb{P}}

\newcommand{\isom}{\cong}

\renewcommand{\epsilon}{\varepsilon}
\renewcommand{\phi}{\varphi}

\newcommand{\OSheaf}{\operatorname{\mathcal O}}

\newcommand{\rien}[1]{}

{\begin{enumerate}\setlength{\itemsep}{#1}}{\end{enumerate}}

{\begin{itemize}\setlength{\itemsep}{#1}}{\end{itemize}}

\CompileMatrices

\title{Rationality of weighted hypersurfaces of special degree}
\author{Michael Chitayat}

\begin{document}

	\maketitle
	\begin{abstract} 
		Let $X \subset \PPP(w_0, w_1, w_2, w_3)$ be a quasismooth well-formed weighted projective hypersurface and let $L = \lcm(w_0,w_1,w_2,w_3)$. We characterize when $X$ is rational under the assumption that $L$ divides $\deg(X)$. Furthermore, we give a new family of normal rational weighted projective hypersurfaces with ample canonical divisor, valid in all dimensions, adding to the list of examples discovered by Koll$\rm{\acute a}$r. Finally, we determine precisely which affine Pham-Brieskorn threefolds are rational, answering a question of Rajendra V. Gurjar. 
	\end{abstract}
	
	\bigskip
	\noindent \textbf{Author information.} Michael Chitayat, University of Ottawa, \textit{mchit007@uottawa.ca}, ORCID: 0000-0002-0590-539X
	
	
	\bigskip
	\noindent \textbf{Statements and declarations.} The author has no competing interests or funding to declare. 
	
	\bigskip
	\noindent \textbf{Acknowledgements.} The author wishes to thank Rajendra V. Gurjar for proposing Question \ref{GurjarQuestion}, Pedro Garcia Sanchez for sharing some useful comments as well as the reference \cite{Brauer1942}, and Daniel Daigle for multiple comments and suggestions, the most important of which significantly shortened the proof of Theorem A.

	\section{Introduction}
	
	Let $X \subseteq \PPP^3$ be a smooth projective hypersurface over $\Comp$ of degree $d$. It is well known that 
	\begin{equation}\label{easyEquiv}
		\text{$X$ is rational} \iff d \leq 3 \iff \text{the geometric genus of $X$ is zero.}
	\end{equation}
	We consider the more general situation of weighted hypersurfaces $X$ in the weighted projective space $\PPP(w_0, w_1, w_2, w_3)$. This additional generality makes understanding when $X$ is rational more difficult; we give new results in this direction with subsequent applications. 
	
	\medskip
	
	A Pham-Brieskorn ring is a $\Comp$-algebra of the form
	$$
	B_{a_0, a_1, \dots,  a_n} = \Comp[X_0, X_1, \dots, X_n] / \lb X_0^{a_0} + X_1^{a_1} + \dots + X_n^{a_n} \rb
	$$
	where each $a_i \in \Nat^+$. These rings, their corresponding affine varieties and their singular points have been studied extensively for decades from many perspectives (see \cite{yoshinaga1974topological}, \cite{hefez1974intersection}, \cite{storch1984picard}, \cite{milnor19753}, etc.). This work was motivated by the following question of Rajendra V. Gurjar, to which we manage to provide a complete answer. 
	\begin{question}\label{GurjarQuestion}
		For which 4-tuples $(a_0, a_1, a_2, a_3)$ is $\Spec B_{a_0, a_1, a_2, a_3}$ rational?
	\end{question}
	
	We begin by defining an $\Nat$-grading on $B_{a_0, a_1, a_2, a_3}$ by declaring that $X_i$ is homogeneous of degree $w_i = L / a_i$, where $L = \lcm(a_0, a_1, a_2, a_3)$. Thus, $\Proj B_{a_0, a_1, a_2, a_3}$ is a hypersurface of the weighted projective space $\PPP(w_0,w_1,w_2,w_3)$. In view of Castelnuovo's Theorem, which states that rationality and unirationality are equivalent in dimension 2 over $\Comp$,
	it is clear that $\Spec B_{a_0, a_1, a_2, a_3}$ is rational if and only if $\Proj B_{a_0, a_1, a_2, a_3}$ is rational.
	As such, we instead answer the following question, which has the same answer as Question \ref{GurjarQuestion}. 
	\begin{question}\label{qProj}
		For which 4-tuples $(a_0, a_1, a_2, a_3)$ is $\Proj B_{a_0, a_1, a_2, a_3}$ rational?
	\end{question}
	
	To answer Question \ref{qProj}, we require a theorem on numerical semigroups. We say that a tuple $(d_1, \dots, d_n) \in (\Nat^+)^{n}$ is \textit{well-formed} if $\gcd(d_1, \dots, d_{i-1}, \hat{d_i}, d_{i+1}, \dots d_n) = 1$ for all $i \in \{1, \dots, n\}$.
	
	\begin{theoremA}\label{TheoremA}
		Suppose $(d_1,d_2,d_3,d_4) \in (\Nat^+)^4$ is well-formed, let $\Gamma = \lb d_1, d_2, d_3, d_4 \rb \subseteq \Nat$ be the numerical semigroup generated by $d_1,d_2,d_3,d_4$ and let $L = \lcm(d_1, d_2, d_3, d_4)$. 
		\begin{enumerate}[\rm(i)]
			\item If $N \geq \max\{ 2L-\sum_{m=1}^4 d_m, 0 \}$ then $N \in \Gamma$. 
			\item  If $n \in \Nat^+$ and $n \cdot L - \sum_{i=1}^4 d_i \in \Nat \setminus \Gamma$, then $n=1$ and $| \{ d_1, d_2 , d_3 , d_4 \} | = 2$.
			
		\end{enumerate}
	\end{theoremA}

	In addition to Theorem A, we prove the rationality of a special family of projective varieties in Proposition \ref{rationalAlldim}.  
	
	\begin{propnonumber}\label{rationalAlldim}
		Let $a,c \in \Nat^+$ and consider the graded polynomial ring $R = \bk_{c,\dots, c, a, \dots, a}[X_1, \dots, X_k, Y_1, \dots, Y_\ell]$ where $\bk$ is a field, $k \geq 1$, $\ell \geq 1$, $\deg(X_i) = c$ and $\deg(Y_j) = a$ for all $i,j$. Suppose $f = g(X_1, \dots, X_k) + h(Y_1, \dots, Y_\ell)$ is irreducible and homogeneous of degree $L = \lcm(a,c)$ and both $g(X_1, \dots, X_k)$ and $h(Y_1, \dots, Y_\ell)$ are non-zero. Then $\Proj(R / \lb f \rb)$ is rational over $\bk$. 
	\end{propnonumber}
	
	Using Theorem A together with Proposition \ref{rationalAlldim}, we eventually obtain the following (Theorem \ref{hypersurfaceDegree}) which generalizes \eqref{easyEquiv}:
	\begin{thmnonumber}
		Let $w_0,w_1,w_2,w_3 \in \Nat^+$ where $w_0 \leq w_1 \leq w_2 \leq w_3$ and let $L = \lcm(w_0,w_1,w_2,w_3)$. Suppose $X = X_f \subset \PPP(w_0,w_1,w_2,w_3)$ is a well-formed quasismooth hypersurface of degree $nL$ for some $n \in \Nat^+$. Then the following are equivalent:
		\begin{enumerate}[\rm(a)]
			\item $X$ is rational,
			\item $h^2(X, \OSheaf_X) = 0$,
			\item one of the following holds
			\begin{enumerate}[\rm(i)]
				\item $n = 1$, $w_0 = w_1$, $w_2 = w_3$ and $\gcd(w_0,w_2) = 1$;
				\item $nL - \sum_{i=0}^3 w_i < 0$.
				
			\end{enumerate} 
		\end{enumerate}
	\end{thmnonumber}
	
	With Theorem \ref{hypersurfaceDegree} available to us, we are able to address Gurjar's Question \ref{qProj}. We will see that Question \ref{qProj} reduces to the special case where $\cotype(a_0, a_1, a_2, a_3) = 0$ (see Definition \ref{cotypeDef}) and hence is completely answered by the following, which is Theorem \ref{PBTHM}. 
	\begin{thmnonumber}
		Suppose $a_0 \leq a_1 \leq a_2 \leq a_3$ and  $\cotype(a_0,a_1,a_2,a_3) = 0$.  Then $\Proj B_{a_0,a_1,a_2,a_3}$ is rational if and only if one of the following holds:
		\begin{enumerate}[\rm(a)]
			\item $a_0 = a_1$, $a_2 = a_3$ and $\gcd(a_0,a_2) = 1$;
			\item $\frac{1}{a_0} + \frac{1}{a_1} + \frac{1}{a_2} + \frac{1}{a_3} > 1$.
		\end{enumerate}
	\end{thmnonumber}
	
	Another interesting application of our results is Example \ref{exampleAmple}, which gives a new family of rational quasismooth weighted hypersurfaces with ample canonical divisor. We believe these are the first such examples since Koll$\rm{\acute a}$r introduced in \cite{Kollar2006} what are now known as \textit{Koll$\acute a$r hypersurfaces}. Our examples are simple and exist in any dimension. As a very special case, if $a,c \in \Nat^+$ are such that $\gcd(a,c) = 1$ and $ac - 2a-2c > 0$ then $\Proj B_{a,a,c,c}$ is a quasismooth normal rational projective surface with quotient singularities and ample canonical divisor.
	
	\section*{Notation}
	
	We use the following terminology and notation throughout the article. 
	
	\begin{itemize}
		\item The set of natural numbers includes $0$ and is denoted by $\Nat$. The set of positive integers is denoted by $\Nat^+$.
		\item Let $\Phi$ denote the set of prime numbers and let $p \in \Phi$. Given $n \in \Nat^+$, define $$v_p(n) = \max\setspec{u \in \Nat}{p^u \text{ divides } n}$$ 
		and observe that $n = \prod_{p \in \Phi} p^{v_p(n)}$ is the prime factorization of $n$. Given $f,g \in \Nat^+$, we define $v_p(\frac{f}{g}) = v_p(f) - v_p(g)$.
		
		\item If $d_1, \dots, d_n \in \Nat^+$, we define $\langle d_1, \dots, d_n \rangle$ to be the submonoid of $(\Nat,+)$ generated by 
		$d_1, \dots, d_n$.
		
		\item A {\it numerical semigroup\/} $\Gamma$ is a subset of $(\Nat,+)$ satisfying the following three conditions:
		\begin{enumerate}[\rm(i)]
			\item $0 \in \Gamma$
			\item $\Nat^+ \setminus \Gamma$ is finite
			\item $x,y \in \Gamma \Rightarrow x + y \in \Gamma$
		\end{enumerate}
		The \textit{Frobenius number} of $\Gamma$ is defined as $F(\Gamma) = \max( \Integ \setminus \Gamma )$. If $\gcd( d_1, \dots, d_n ) = 1$, we may write $F( d_1, \dots, d_n )$ instead of $F(\lb d_1, \dots, d_n \rb )$.

		\item Let $\bk$ be a field. If $K / \bk$ is a field extension, the notation $K = \bk^{(n)}$ means that $K / \bk$ is purely transcendental of transcendence degree $n$. 
		
		\item A \textit{variety} is an integral separated scheme of finite type over an algebraically closed field $\bk$. A \textit{surface} is a two-dimensional variety. 
		
		\item Let $X$ be a normal variety over $\bk$. Then $\Div(X)$ is its group of Weil divisors and $K_X$ denotes a canonical divisor of $X$. The group of $\Rat$-divisors is denoted by $\Div(X, \Rat)$. A divisor $D \in \Div(X)$ is ample if there exists some $n \in \Nat^+$ such that $\OSheaf_X(nD)$ is a very ample invertible sheaf (over $\bk$).  
		
		\item A \textit{del Pezzo surface} is a normal projective surface with at most quotient singularities such that $-K_X$ is ample. It is well known that a del Pezzo surface over $\Comp$ is rational.  
		
	\end{itemize}
	
	\section{A Semigroup Theorem}\label{Sec:SemigroupStatement}
	
	We state a theorem used in Section \ref{Sec:Rationality}. The proof of Theorem A is given in Section \ref{Sec:Numerical}. 
	
	\begin{definition}\label{wellformedtuple}
		A tuple $(d_1, \dots, d_n) \in (\Nat^+)^{n}$ is \textit{well-formed} if $\gcd(d_1, \dots, d_{i-1}, \hat{d_i}, d_{i+1}, \dots d_n) = 1$ for all $i \in \{1, \dots, n\}$.
	\end{definition}
	
	\begin{theoremA}\label{TheoremA}
		Suppose $(d_1,d_2,d_3,d_4) \in (\Nat^+)^4$ is well-formed, let $\Gamma = \lb d_1, d_2, d_3, d_4 \rb$ and let $L = \lcm(d_1, d_2, d_3, d_4)$. 
		\begin{enumerate}[\rm(i)]
			\item If $N \geq \max\{ 2L-\sum_{m=1}^4 d_m, 0 \}$ then $N \in \Gamma$. 
			\item  If $n \in \Nat^+$ and $n \cdot L - \sum_{i=1}^4 d_i \in \Nat \setminus \Gamma$, then $n=1$ and $| \{ d_1, d_2 , d_3 , d_4 \} | = 2$.
			
		\end{enumerate}
		
	\end{theoremA} 
	
	\begin{remark} \label{easyApplication}
		With the assumptions of Theorem A(ii), if we assume in addition that $d_1 \leq d_2 \leq d_3 \leq d_4$, then $d_1 = d_2, d_3 = d_4$ and $\gcd(d_1,d_3) = 1$. 
	\end{remark}
	
	\begin{remark}\label{semigroupComments}
		Not only is the well-formedness assumption in Theorem A natural given the geometric context of weighted projective varieties discussed in the following section, it is also the case that various general questions about numerical semigroups can be reduced to the special case where $(d_1, \dots, d_n)$ is well-formed. (See for instance Section 3 of \cite{johnson1960linear}, Proposition 8 in \cite{froberg1986numerical}, as well as Lemma 2.16 and Proposition 2.17 in \cite{rosales2009numerical}.)
	\end{remark}

	\section{Rationality of Some Weighted Hypersurfaces}\label{Sec:Rationality}

	\begin{nothing}\label{weightedComments}
		Let $S = \Comp_{w_0, \dots, w_n}[X_0, \dots, X_n]$ denote the graded polynomial ring where $n \geq 1$ and $\deg(X_i) = w_i$ for each $i = 0, \dots, n$. The weighted projective space $\PPP = \PPP(w_0, \dots , w_n) = \Proj(S)$ is \textit{well-formed} if $(w_0, \dots , w_n)$ is well-formed. Every weighted projective space is a projective variety and is isomorphic to a well-formed weighted projective space. Note also that every weighted projective space is rational.  
		
		\begin{nothing}\label{weightedCompleteIntersection}
			Let $I$ be a homogeneous prime ideal of the graded ring $S = \Comp_{w_0,\dots,w_n}[X_0, \dots, X_n]$ and define $X_I = \Proj (S/I) \subseteq \PPP(w_0, \dots, w_n)$. If $f_1, \dots, f_k$ are homogeneous elements of $S$, we abbreviate $X_{\lb f_1, \dots, f_k \rb}$ by $X_{f_1, \dots, f_k}$. If $I$ is generated by a regular sequence $(f_1, \dots , f_k)$ of homogeneous elements of $S$ of respective degrees  $d_i = \deg(f_i)$ where $i = 1, \dots , k$, then $X_I$ is called a \textit{weighted complete intersection} of \textit{multidegree $(d_1, \dots, d_k)$}. In particular, if $k=1$, $f_1 = f$ and $d_1 = d$, we will say that $X_I = X_f$ is a \textit{weighted hypersurface} of \textit{degree $d$}. The closed subset $C_{X_I} = V(I) \subseteq \aff^{n+1}$ is called the  \textit{affine cone over $X_I$};	note that $C_{X_I}$ passes through the origin of $\aff^{n+1}$ and that $C_{X_I} \isom \Spec( S/ I )$ is an integral affine scheme. The variety $X_I$ is \textit{quasismooth} if $C_{X_I}$ is nonsingular away from the origin. A weighted complete intersection $X_I \subseteq \PPP(w_0, \dots, w_n)$ is \textit{well-formed} if both
			\begin{enumerate}[\rm(i)]
				\item $\PPP(w_0, \dots, w_n)$ is well-formed,
				\item $\codim_{X_I}(X_I \cap \Sing(\PPP)) \geq 2$.
			\end{enumerate}
		\end{nothing}
		
		The \textit{amplitude} of a well-formed weighted complete intersection $X_{f_1, \dots, f_k} \subseteq \PPP(w_0, \dots, w_n)$ is the quantity $\alpha = \sum_{i = 1}^k d_i - \sum_{j = 0}^n w_j$ where $d_i = \deg(f_i)$. (We will preserve this notation for the amplitude $\alpha$ through to the end of Section \ref{Sec:Rationality}.) It is well known that if $X_I$ is a well-formed quasismooth weighted complete intersection, then $X_I$ is a normal, Cohen-Macaulay, $\Rat$-factorial variety with at most cyclic quotient (hence rational) singularities.  
	\end{nothing}
	
	We make use of \ref{ramification} to \ref{wellformedIFFSatCodim1} in the proof of Corollary \ref{ampleRational}. Paragraph   \ref{ramification} appears in Section 5.4 of \cite{ChitayatDaigleCylindricity},  Lemma \ref{BetterDemazure} (a) is due to Demazure. Lemma \ref{BetterDemazure} (b) and Proposition \ref{wellformedIFFSatCodim1} are unpublished results of Daigle.
	
	\begin{nothing}\label{ramification}
		If $B = \bigoplus_{i\in\Nat} B_i$ is an $\Nat$-graded domain, the number $e(B) = \gcd\setspec{ i \in \Nat }{ B_i \neq 0 }$ is called the
		{\it saturation index\/} of $B$. One says that $B$ is {\it saturated in codimension $1$} if $e(B/\pgoth)=e(B)$ for every height $1$ homogeneous prime ideal $\pgoth$ of $B$. Let $X = \Spec B$ and $Y = \Proj B$. Let $Y^{(1)}$ be the set of height $1$ homogeneous prime ideals of $B$.
		If $\pgoth \in Y^{(1)}$ then $B_{(\pgoth)} \subset B_\pgoth$ is an extension of discrete valuation rings whose ramification index we denote by $e_\pgoth$.
		For each $\pgoth \in Y^{(1)}$, let $C_\pgoth^Y$ (resp.\  $C_\pgoth^X$) be the closure of $\{ \pgoth \}$ in $Y$ (resp.\ in $X$);
		note that $C_\pgoth^Y$ (resp.\  $C_\pgoth^X$) is a prime divisor of $Y$ (resp.\ of $X$),
		and that each prime divisor on $Y$ is a $C_\pgoth^Y$ for some $\pgoth \in Y^{(1)}$. For each prime divisor $C_\pgoth^Y \in Y^{(1)}$, define $\phi(C_\pgoth^Y) = e_\pgoth C_\pgoth^{X}$, and extend linearly to a $\Rat$-linear map $\phi : \Div(Y, \Rat) \to \Div(X,\Rat)$, where $D \mapsto \phi(D)$.
	\end{nothing}
	
	\begin{lemma} \label{BetterDemazure}
		Let $B$ be an $\Nat$-graded normal domain that is finitely generated over $\Comp$ such that $e(B)=1$ and $\height(B_+)>1$. 
		\begin{enumerate}[\rm(a)]
			
			\item There exists an ample $\Rat$-divisor $D$ of $Y = \Proj B$ such that $\OSheaf_Y(n) = \OSheaf_Y(nD)$ for all $n \in \Integ$.
			
			\item If $B$ is saturated in codimension $1$, then $D \in \Div(Y)$ (i.e. $D$ has integral coefficients).
			
		\end{enumerate}
	\end{lemma}
	
	\begin{proof}
		We use the notation of \ref{ramification}. Let $f,g$ be nonzero homogeneous elements of $B$ such that $\deg(f)-\deg(g)=1$ and let $W = \div_X(f/g) \in \Div(X)$. By the Theorem below Section 3.5 in \cite{Demazure},
		there exists a unique $D \in \Div(Y,\Rat)$ such that $\phi(D) = W$ and this $D$ is ample and satisfies $\OSheaf_X(n D) \isom \OSheaf_X(n)$ for all $n \in \Integ$. This shows (a).
		
		If $B$ is saturated in codimension $1$, then Corollary 9.4 of \cite{daigle2023rigidity} implies that $e_\pgoth=1$ for all $\pgoth \in Y^{(1)}$;
		since $W$ has integral coefficients, it follows that $D$ has integral coefficients, proving (b). 
	\end{proof}
	
	\begin{proposition}\label{wellformedIFFSatCodim1}
		Let $n\ge2$, let $\proj = \proj(w_0,\dots,w_n)$ be a well-formed weighted projective space
		and let $I$ be a homogeneous prime ideal of $R = \Comp_{w_0,\dots,w_n}[X_0,\dots,X_n]$ with $\height I < n$.
		Consider the graded ring $B = R/I$ and the closed subvariety $X = V_+(I)$ of $\proj$.
		Then $X$ is well-formed if and only if $B$ is saturated in codimension $1$.
	\end{proposition}
	
	\begin{proof}
		Let $\pi : R \to B$ be the canonical epimorphism. Let $w=\prod_{i=0}^n w_i$. For each prime factor $p$ of $w$, let $J_p = \setspec{ j }{ 0 \le j \le n \text{ and } p \nmid w_j }$ and $\qgoth_p = \sum_{j \in J_p} R X_j$.
		It is well known that $\Sing\proj = \bigcup_{ p \mid w } V_+(\qgoth_p)$.
		Also note that
		\begin{equation}  \label {uy8qu3wy5trdqi}
			\text{given any homogeneous prime ideal $\hgoth$ of $R$, \quad $p \mid e(R/\hgoth)$ $\Leftrightarrow$ $\qgoth_p \subseteq \hgoth$.}
		\end{equation}
		
		Suppose that $B$ is not saturated in codimension $1$.
		Then there exists a height $1$ homogeneous prime ideal $\pgoth$ of $B$ such that $e(B / \pgoth) \neq 1$.
		Let $\hgoth = \pi^{-1}(\pgoth)$; then $R/\hgoth \isom B/\pgoth$, so $e(R/\hgoth) = e(B/\pgoth) \neq 1$.
		Choose a prime factor $p$ of $e(R/\hgoth)$; then $p\mid w$ and (by \eqref{uy8qu3wy5trdqi}) $\qgoth_p \subseteq \hgoth$.
		We also have $I \subseteq \hgoth$, so $V_+(\hgoth) \subseteq V_+(I) \cap V_+(\qgoth_p) \subseteq X \cap \Sing\proj$.
		Since $\height\hgoth = 1 +\height I$, we find $\dim V_+(\hgoth) = \dim X - 1$, so $\codim_X( X \cap \Sing\proj ) \le 1$ and hence $X$ is not well-formed. 
		
		Conversely, suppose that $X$ is not well-formed.
		Then $\codim_X( X \cap \Sing\proj ) \le 1$, so there exists a homogeneous prime ideal $\hgoth$ of $R$
		such that $V_+(\hgoth) \subseteq X \cap \Sing \proj$ and $\dim V_+(\hgoth) = \dim X - 1$.
		Note that $I \subseteq \hgoth$ and that $\height\hgoth = \height I + 1$; so $\pgoth = \pi(\hgoth)$ is a homogeneous prime ideal of $B$ of height $1$.
		Since  $V_+(\hgoth) \subseteq \Sing \proj = \bigcup_{p\mid w}V_+(\qgoth_p)$, there exists a prime divisor $p$ of $w$ such that $V_+(\hgoth) \subseteq V_+(\qgoth_p)$
		and hence $\qgoth_p \subseteq \hgoth$.
		So \eqref{uy8qu3wy5trdqi} implies that $p$ divides $e(R/\hgoth) = e(B/\pgoth)$, which implies that $B$ is not saturated in codimension $1$.
	\end{proof}
	
	We generalize an unpublished result of Michela Artebani. See Proposition 4.3.4 in  \cite{ChitayatThesis}.

	\begin{proposition}\label{rationalAlldim}
		Let $a,c \in \Nat^+$ and consider the graded polynomial ring $R = \bk_{c,\dots, c, a, \dots, a}[X_1, \dots, X_k, Y_1, \dots, Y_\ell]$ where $\bk$ is a field, $k \geq 1$, $\ell \geq 1$, $\deg(X_i) = c$ and $\deg(Y_j) = a$ for all $i,j$. Suppose $f = g(X_1, \dots, X_k) + h(Y_1, \dots, Y_\ell)$ is irreducible and homogeneous of degree $L = \lcm(a,c)$ and both $g(X_1, \dots, X_k)$ and $h(Y_1, \dots, Y_\ell)$ are non-zero. Then $\Proj(R / \lb f \rb)$ is rational over $\bk$. 
	\end{proposition}
	
	\begin{proof}
		Let $B = R / \lb f \rb$ and write $B = \bk[x_1, \dots x_k, y_1, \dots y_\ell]$ where $x_i$ and $y_j$ are the canonical images of $X_i$ and $Y_j$ in $B$. Let $g = \gcd(a,c)$. Since $\Proj B \isom \Proj B^{(g)}$, we may assume without loss of generality that $\gcd(a,c) = 1$ and that $L = ac$. We define the following elements of the function field $K$ of $\Proj B$:
		\begin{equation*} t_i = \frac{x_i}{x_k} \text{ for each } i = 1, \dots, k-1, \quad u = \frac{x_k^a}{y_\ell^c}, \quad v_j = \frac{y_j}{y_\ell} \text{ for each } j = 1, \dots, \ell-1.
		\end{equation*}
		We claim that 
		\begin{equation}\label{someinclusions}
			B_{(y_\ell)} \subseteq \bk(t_1, \dots, t_{k-1}, u , v_1, \dots, v_{\ell-1}) \subseteq K
		\end{equation}
		where the second inclusion is obvious. To prove the first inclusion, observe that $B_{(y_\ell)}$ is generated by elements of the form
		\begin{equation}\label{monomials}
			\frac{x_1^{\alpha_1}\cdots x_k^{\alpha_k}y_1^{\beta_1} \cdots y_{\ell-1}^{\beta_{\ell - 1}}}{y_\ell^{\beta_\ell}} \text{  such that $\alpha_i, \beta_j \in \Nat$ and $c \sum_{i = 1}^k \alpha_i + a \sum_{j = 1}^{\ell - 1} \beta_j = a \beta_\ell$} .
		\end{equation}
		Consider any element from \eqref{monomials}. Since $\gcd(a,c) = 1$, $a \mid \sum_{i = 1}^k \alpha_i$. So, $\sum_{i = 1}^k \alpha_i = qa$ for some $q \in \Nat^+$ and $\alpha_k = qa - \sum_{i = 1}^{k-1}\alpha_i$. We then have 
		$$
		\beta_\ell = \frac{c\sum_{i = 1}^k \alpha_i}{a} + \sum_{j = 1}^{\ell - 1} \beta_j = \frac{(c\sum_{i = 1}^{k-1} \alpha_i) + c\alpha_k }{a}  + \sum_{j = 1}^{\ell - 1} \beta_j = \frac{(c\sum_{i = 1}^{k-1} \alpha_i) + cqa - c(\sum_{i = 1}^{k-1} \alpha_i) }{a}  + \sum_{j = 1}^{\ell - 1} \beta_j = cq + \sum_{j = 1}^{\ell - 1} \beta_j.
		$$
		We have 
		\begin{align*}
			\frac{x_1^{\alpha_1} \cdots x_k^{\alpha_k} y_1^{\beta_1} \dots y_{\ell-1}^{\beta_{\ell - 1}}}{y_\ell^{cq + \sum_{j = 1}^{\ell - 1}\beta_j}} &= \left(\frac{x_1}{x_k} \right)^{\alpha_1} \dots \left(\frac{x_{k-1}}{x_k}\right)^{\alpha_{k-1}} \left(\frac{x_k^{\sum_{i = 1}^k \alpha_i}}{y_\ell^{cq}} \right) \left(\frac{y_1}{y_\ell}\right)^{\beta_1} \dots \left(\frac{y_{\ell -1}}{y_\ell}\right)^{\beta_{\ell-1}}  \\
			&=  \left(\frac{x_1}{x_k} \right)^{\alpha_1} \dots \left(\frac{x_{k-1}}{x_k}\right)^{\alpha_{k-1}} \left(\frac{x_k^{qa}}{y_\ell^{cq}} \right) \left(\frac{y_1}{y_\ell}\right)^{\beta_1} \dots \left(\frac{y_{\ell -1}}{y_\ell}\right)^{\beta_{\ell-1}} \\
			& = t_1^{\alpha_1} \dots t_{k-1}^{\alpha_{k-1}} u^q v_1^{\beta_1} \dots v_{\ell - 1}^{\beta_{\ell - 1}}.  
		\end{align*}
		This shows that every element from \eqref{monomials} is an element of $\bk(t_1, \dots, t_{k-1}, u, v_1, \dots, v_{\ell - 1})$ and consequently \eqref{someinclusions} is true. Since $\Frac{B_{(y_\ell)}} = K$, it follows from \eqref{someinclusions} that $K = \bk(t_1, \dots, t_{k-1}, u, v_1, \dots,  v_{\ell - 1})$. Since $\trdeg_\bk K = k + \ell -2$, it now suffices to show that $u \in \bk(t_1, \dots, t_{k-1}, v_1, \dots,  v_{\ell - 1})$. 
		
		Let $M = \setspec{(\mu_1, \dots, \mu_k) \in \Nat^k}{\sum_{i=1}^k \mu_i = a}$ and let $N = \setspec{(\nu_1, \dots, \nu_\ell) \in \Nat^\ell}{\sum_{i=1}^\ell \nu_i = c}$. Write 
		$$g(X_1, \dots, X_k) = \sum_{(\mu_1, \dots, \mu_k) \in M} a_m X_1^{\mu_1} \dots X_k^{\mu_k}  \text{ and } h(Y_1, \dots, Y_\ell) = \sum_{(\nu_1, \dots, \nu_\ell) \in N} b_n Y_1^{\nu_1} \dots Y_\ell^{\nu_\ell}$$ where $m = (\mu_1, \dots, \mu_k) \in M$, $n = (\nu_1, \dots, \nu_\ell) \in N$ and the coefficients $a_m \in \bk$ (resp $b_n \in \bk$) are not all zero. Then
		\begin{align*}
			0 &= g(x_1, \dots, x_k) + h(y_1, \dots, y_\ell) = \sum_{(\mu_1, \dots, \mu_k) \in M} a_m x_1^{\mu_1} \dots x_k^{\mu_k} + \sum_{(\nu_1, \dots, \nu_\ell) \in N} b_n y_1^{\nu_1} \dots y_\ell^{\nu_\ell} \\
			0 &= \sum_{(\mu_1, \dots, \mu_k) \in M}\frac{ a_m x_1^{\mu_1} \dots x_{k-1}^{\mu_{k-1}} x_k^{\mu_k} x_k^{\sum_{i=1}^{k-1}\mu_i}}{x_k^{\sum_{i=1}^{k-1}\mu_i} y_\ell^c} + \sum_{(\nu_1, \dots, \nu_\ell) \in N} \frac{b_n y_1^{\nu_1} \dots y_\ell^{\nu_\ell}}{y_\ell^c} \\
			& = \sum_{(\mu_1, \dots, \mu_k) \in M}\frac{ a_m x_1^{\mu_1} \dots x_{k-1}^{\mu_{k-1}} x_k^a}{x_k^{\sum_{i=1}^{k-1}\mu_i} y_\ell^c} + \sum_{(\nu_1, \dots, \nu_\ell) \in N} \frac{b_n y_1^{\nu_1} \dots y_\ell^{\nu_\ell}}{y_\ell^c} \\
			& = u \sum_{(\mu_1, \dots, \mu_k) \in M} a_m t_1^{\mu_1} \dots t_{k-1}^{\mu_{k-1}} + \sum_{(\nu_1, \dots, \nu_\ell) \in N} b_n v_1^{\nu_1} \dots v_{\ell-1}^{\nu_{\ell-1}}
		\end{align*}
		from which it follows that $$u = \frac{-\sum_{(\nu_1, \dots, \nu_\ell) \in N} b_n v_1^{\nu_1} \dots v_{\ell-1}^{\nu_{\ell-1}}}{\sum_{(\mu_1, \dots, \mu_k) \in M} a_m t_1^{\mu_1} \dots t_{k-1}^{\mu_{k-1}}}$$
		noting that the denominator is non-zero because not all the $a_m$ are zero and $t_1, \dots, t_{k-1}$ are algebraically independent over $\bk$. It follows that $u \in \bk(t_1, \dots, t_{k-1}, v_1, \dots,  v_{\ell - 1})$ which completes the proof.
		
	\end{proof}
	
	\begin{lemma}\label{mn0}
		Let $a,c \in \Nat^+$ and suppose $\gcd(a,c) = 1$. If $m,n\in \Nat$ satisfy $ma + nc = ac$, then either $m = 0$ or $n = 0$. 
	\end{lemma}
	\begin{proof}
		Assume $ma + nc = ac$. Then $ma = (a-n)c$ so $c \mid ma$. Since $\gcd(a,c) = 1$, $c \mid m$. So $m = 0$ or $m = c$, in which case $n = 0$. 
	\end{proof}
	
	\begin{corollary}\label{rationalDim2}
		Let $a,c \in \Nat^+$ be such that $\gcd(a,c) = 1$ and let $B = \Comp_{c,c,a,a}[X_0,X_1,X_2,X_3] / \lb f \rb$ where $f$ is irreducible and homogeneous of degree $ac$. Then $\Proj B$ is rational over $\Comp$. 
	\end{corollary}
	
	\begin{proof}
		By Lemma \ref{mn0}, $f = g(X_0,X_1) + h(X_2,X_3)$ where each of $g(X_0,X_1)$ and $h(X_2,X_3)$ is either zero or is homogeneous of degree $ac$. We write $B = \Comp[x_0, x_1, x_2, x_3]$ where $x_i$ is the canonical image of $X_i$ in $B$. Let $g = g(X_0,X_1)$, $h = h(X_2,X_3)$ and let $\bar{f}, \bar{g}$ and $\bar{h}$ denote the images of $f,g$ and $h$ in $B$. 
		
		If $\bar{g} = 0$, then $f \mid g$ so $f = g$. Since $f$ is irreducible, $ f = \lambda_0 X_0 + \lambda_1 X_1$ where $\lambda_0,\lambda_1 \in \Comp$ are not both 0. We then obtain that $\Proj B \isom \PPP(c,a,a) \isom \PPP(c,1,1)$ which is rational. The same argument shows that $\Proj B$ is rational if $\bar{h} = 0$. Assume now that $\bar{g} \neq 0$ and $\bar{h} \neq 0$. Then $g(X_0,X_1)$ and $h(X_2,X_3)$ are nonzero and the result follows from Proposition \ref{rationalAlldim}.  
	\end{proof}

	We state Lemma 7.1 of \cite{iano-fletcher_2000}. Note that the last line should read $A_{-(k-\alpha)}$ instead of $A_{-k-\alpha}$ as in \cite{iano-fletcher_2000}.
	
	\begin{lemma}
		\label{FletcherCohomology}
		Let $X = X_{f_1, \dots, f_r} \subseteq \PPP(w_0, \dots, w_n)$ be a well-formed quasismooth weighted complete intersection. Let $A$ be the graded ring $\Comp_{w_0, \dots, w_n}[X_0, \dots, X_n] / \lb f_1, \dots, f_r \rb$. Then,

		\[
		H^i(X, \OSheaf_{X}(k)) \isom 
		\begin{cases}
			A_k & \text{if } i = 0\\
			0  & \text{if } 1 \leq i < \dim X \\
			A_{-(k-\alpha)} & \text{if } i = \dim X \\
		\end{cases}
		\]
		for all $k \in \Integ$. 
	\end{lemma}
	
	\begin{proposition}\label{notUniruled}
		Let $X = X_{f_1, \dots, f_r} \subseteq \PPP(w_0, \dots, w_n)$ be a well-formed quasismooth weighted complete intersection of dimension $q$ that is not contained in a hyperplane. If $\alpha$ belongs to the submonoid $\lb w_0, \dots, w_n \rb$ of $(\Nat, +)$, then $h^q(X, \OSheaf_X) > 0$. Moreover, $X$ is not uniruled and not rational. 
	\end{proposition}
	
	\begin{proof}
		Let $A$ be the graded ring $\Comp_{w_0, \dots, w_n}[X_0, \dots, X_n] / \lb f_1, \dots, f_r \rb$ and let $\tilde{X}$ be a resolution of singularities of $X$. Suppose $\alpha \in \lb w_0, \dots, w_n \rb$. If $\alpha = 0$, then $A_\alpha = \Comp$. Otherwise,  $\alpha > 0$ and since $X$ is not contained in a hyperplane, $A_\alpha$ contains a nonzero monomial of $A$. We obtain
		$$0 \neq A_\alpha \isom H^q(X, \OSheaf_X) \isom H^q(\tilde{X}, \OSheaf_{\tilde{X}}),$$ 
		the first isomorphism by Lemma \ref{FletcherCohomology}, the second since $X$ has rational singularities. This shows that $0 < h^q(X, \OSheaf_X) = h^q(\tilde{X}, \OSheaf_{\tilde{X}})$. Since $\tilde{X}$ has non-zero genus, it has non-negative Kodaira dimension, and hence is not uniruled. It follows that $X$ is not uniruled and not rational.
	\end{proof}
	
	\begin{example}
		Proposition \ref{notUniruled} may fail if $X$ is contained in a hyperplane. Let $f_1 = X_0^3 + X_1^3 + X_2^7 + X_3^7$, $f_2 = X_4$ and let $X = X_{f_1, f_2} \subset \PPP(7,7,3,3,1)$. Then $X$ is quasismooth and well-formed, the amplitude $\alpha = 1$ belongs to the submonoid $\lb 7,7,3,3,1 \rb$ of $\Nat^+$ but $X \isom \Proj B_{3,3,7,7}$ is rational by Proposition \ref{rationalAlldim}. 
	\end{example}
	
	\begin{remark}\label{PSRemark}
		In Proposition 2.13 (ii) of \cite{przyjalkowski2021automorphisms}, the authors prove that if $X = X_{f_1, \dots, f_r} \subseteq \PPP(w_0, \dots, w_n)$ is a quasismooth well-formed weighted complete intersection,  $\alpha > 0$ and $w_i \mid \alpha$ for all $i \in \{0,\dots, n\}$, then $X$ is not uniruled. They then remark that they do not know whether the assumption that $w_i \mid \alpha$ for all $i$ is necessary for their result to hold. Proposition \ref{notUniruled} shows that their assumption can often be relaxed. 
	\end{remark}
	
	\begin{corollary} \label{prelimRationality}
		Suppose $X = X_f \subset \PPP(w_0,w_1,w_2,w_3)$ is a well-formed quasismooth weighted hypersurface of degree $d \in \Nat^+$. 
		\begin{enumerate}[\rm(a)]
			\item If $\alpha < 0$, then $X$ is rational. 
			\item If $\alpha$ belongs to the submonoid $\lb w_0, w_1, w_2, w_3 \rb$ of $(\Nat, +)$ then $h^2(X,\OSheaf_X) > 0$ and $X$ is not rational.     
			\item If $d = \sum_{i=0}^3 w_i$ or $d \geq \max\{2L,  \sum_{i=0}^3 w_i\}$ where $L = \lcm(w_0,w_1,w_2,w_3)$, $X$ is not rational.
		\end{enumerate}
	\end{corollary}
	
	\begin{proof}
		For (a), it is well-known that $X$ is a del Pezzo surface with quotient singularities (see p.790 of \cite{Cheltsov2008ExceptionalDP}) and so $X$ is rational. For (b), we note that if $\alpha \in \lb w_0, w_1, w_2, w_3 \rb$ then $\alpha \geq 0$ and so $X$ is not contained in a hyperplane and the result follows from Proposition \ref{notUniruled}. Part (c) follows from Theorem A(i) and part (b). 
	\end{proof}
	
	\begin{theorem}\label{hypersurfaceDegree}
		Let $w_0,w_1,w_2,w_3 \in \Nat^+$ where $w_0 \leq w_1 \leq w_2 \leq w_3$ and let $L = \lcm(w_0,w_1,w_2,w_3)$. Suppose $X = X_f \subset \PPP(w_0,w_1,w_2,w_3)$ is a well-formed quasismooth hypersurface of degree $nL$ for some $n \in \Nat^+$. Then the following are equivalent:
		\begin{enumerate}[\rm(a)]
			\item $X$ is rational,
			\item $h^2(X, \OSheaf_X) = 0$,
			\item one of the following holds
			\begin{enumerate}[\rm(i)]
				\item $n = 1$, $w_0 = w_1$, $w_2 = w_3$ and $\gcd(w_0,w_2) = 1$;
				\item $nL - \sum_{i=0}^3 w_i < 0$.
				
			\end{enumerate} 
		\end{enumerate}
		
	\end{theorem}
	
	\begin{proof}
		
		Let $\Gamma = \langle w_0, w_1, w_2, w_3 \rangle$, let $X = X_f$, and recall that $\alpha = nL - \sum_{i=0}^3 w_i$. Since $X$ is well-formed, we have $\gcd\setspec{ w_i }{ i \in I } = 1$ for every subset $I$ of $\{0,1,2,3\}$ of cardinality $3$.
		
		That (a) implies (b) follows from the fact that $X$ has rational singularities. Assume (b) holds. Then Corollary \ref{prelimRationality} (b) implies that $\alpha \notin \Gamma$, so either $\alpha < 0$ or $\alpha \in \Nat \setminus \Gamma$. If $\alpha < 0$, (ii) holds. If $\alpha \in \Nat \setminus \Gamma$, Theorem A(ii) implies that $n=1$ and $| \{ w_0, w_1, w_2, w_3 \} | = 2$, so (i) holds. This proves that (b) implies (c). We show that (c) implies (a). If (i) holds, then $X$ is rational by Corollary \ref{rationalDim2}. If (ii) holds, then $X$ is rational by Corollary \ref{prelimRationality} (a). 
		
	\end{proof}
	
	\begin{remark}
		As discussed in Remark \ref{semigroupComments}, Proposition \ref{notUniruled} and Theorem \ref{hypersurfaceDegree} further motivate the study of numerical semigroups generated by well-formed tuples. 
	\end{remark}
	
	\begin{nothing}\label{ampleDiscussion}
		Let $X = X_f \subseteq \PPP(w_0, \dots, w_n)$ be a well-formed quasismooth hypersurface. We already noted that $X_f$ is a normal projective variety with cyclic quotient singularities.
		Since $X$ is Cohen-Macaulay, its canonical sheaf coincides with its dualizing sheaf, so Theorem 3.3.4 of \cite{dolgachev} implies that
		$\OSheaf_X(K_X) \isom \OSheaf_X(\alpha)$, where $\alpha$ is the amplitude of $X$. Note that the affine cone $C_X$ is a hypersurface of $\aff^n$ with only one singular point; so $C_X$ and its coordinate ring which we denote by $B$,  are normal. Since $X$ is well-formed, Proposition \ref{wellformedIFFSatCodim1} implies that $B$ is saturated in codimension one and Lemma \ref{BetterDemazure} implies that there exists an ample integral divisor $D \in \Div(X)$ such that $\OSheaf_X(nD) \isom \OSheaf_X(n)$ for all $n \in \Integ$. This implies that $\OSheaf_X(\alpha D) \isom \OSheaf_X(\alpha) \isom \OSheaf_X(K_X)$. It follows that $\lfloor \alpha D \rfloor$ is a canonical divisor of $X$; since $D$ is integral, $\lfloor \alpha D \rfloor = \alpha D$, so $\alpha D$ is a canonical divisor of $X$.
	\end{nothing}
	
	\begin{corollary}\label{ampleRational}
		Let $w_0,w_1,w_2,w_3 \in \Nat^+$ where $w_0 \leq w_1 \leq w_2 \leq w_3$ and let $L = \lcm(w_0,w_1,w_2,w_3)$.	Suppose $X_f \subset \PPP(w_0,w_1,w_2,w_3)$ is a well-formed quasismooth hypersurface of degree $nL$ for some $n \in \Nat^+$. The following are equivalent:
		\begin{enumerate}[\rm(a)]
			\item $X_f$ is rational and has ample canonical divisor;
			\item $n = 1$, $w_0 = w_1$, $w_2 = w_3$, $\gcd(w_0,w_2) = 1$, and $L > 2w_0 + 2w_2$.
		\end{enumerate}
	\end{corollary}
	\begin{proof}
		Let $X = X_f$. By \ref{ampleDiscussion}, 
		there exists an ample integral divisor $D \in \Div(X)$ such that $K_X = \alpha D$ where $\alpha = nL - \sum_{i=0}^3 w_i$. Suppose (a) holds. Then $K_X = \alpha D$ is ample, so $\alpha = nL - \sum_{i=0}^3 w_i > 0$, so condition (ii) of Theorem \ref{hypersurfaceDegree} is not satisfied.
		Since $X$ is rational, condition (i) of Theorem \ref{hypersurfaceDegree} must be satisfied, so (b) holds. Conversely, if (b) holds then $X$ is rational by Theorem \ref{hypersurfaceDegree}, and $\alpha = L -2w_0 - 2w_2 > 0$,
		so $K_X = \alpha D$ is ample.
	\end{proof}

	\begin{example}\label{exampleAmple}
		For any $n \geq 2$, we can construct a normal rational projective $\Comp$-variety of dimension $n$ with quotient singularities and ample canonical divisor. Indeed, let $\bk = \Comp$ and choose some $f$ satisfying the assumptions of Proposition \ref{rationalAlldim} such that $X_f = \Proj (S / \lb f \rb) \subset \PPP(c, \dots, c, a, \dots, a)$ is well-formed and quasismooth, and such that the amplitude $\alpha = ac - kc - \ell a > 0$. By \ref{ampleDiscussion}, the weighted hypersurface $X_f$ is a normal projective variety with quotient singularities and ample canonical divisor and by Proposition \ref{rationalAlldim}, $X_f$ is rational. In particular, the variety $X_f = \Proj B_{3,3,7,7} \subset \PPP(7,7,3,3)$ is one such example.   By Theorem 8.1 in \cite{iano-fletcher_2000}, one can produce many such examples.  
	\end{example}
	
	\begin{remark}\label{KollarHypersurfaces}
		In Section 5 of \cite{Kollar2006}, Koll$\rm{\acute a}$r defines a family of weighted hypersurfaces $H(a_1, \dots, a_n) \subseteq \PPP(a_1, \dots, a_n)$, now known as \textit{Koll$\acute a$r hypersurfaces}; some of these (see Theorem 39 in \cite{Kollar2006}) are rational with ample canonical divisor. As far as we know, they represent the first family of rational weighted projective hypersurfaces with ample canonical divisor, with Example \ref{exampleAmple} providing another such family. (Currently, we only know of these two families with this property.) Note the special case of 2-dimensional Koll$\rm{\acute a}$r hypersurfaces is studied in detail in  \cite{urzua2018characterization}.     
	\end{remark}
	
	In view of Theorem \ref{hypersurfaceDegree} and Remark \ref{KollarHypersurfaces}, we note that we know of no examples of two-dimensional non-rational quasismooth weighted hypersurfaces $X$ such that $h^2(X, \OSheaf_X) = 0$. Hence we ask the following question: 
	
	\begin{question}
		Does there exist a non-rational two-dimensional well-formed quasismooth hypersurface $X$ such that $h^2(X, \OSheaf_{X}) = 0$? Theorem \ref{hypersurfaceDegree} implies that if such an $X$ exists and is a hypersurface, the degree of $f$ is strictly positive and not divisible by $L$.   
	\end{question}

	\section{The Rational Affine Pham-Brieskorn Threefolds}\label{RationalAffine}
	
	Let $B_{a_0,\dots, a_n} = \Comp[X_0,\dots ,X_n]/ \lb X_0^{a_0} + X_1^{a_1} + \dots + X_n^{a_n} \rb$. The ring $B_{a_0,\dots, a_n}$ is called a \textit{Pham-Brieskorn ring} and $ \Spec B_{a_0,\dots, a_n}$ is called an \textit{affine Pham-Brieskorn variety}. We apply the results of the previous sections to answer the following question of R.V. Gurjar:
	
	\begin{question}\label{rationalQuestion}
		For which 4-tuples $(a_0,a_1,a_2,a_3)$ is $\Spec B_{a_0, a_1,a_2,a_3}$ a $\Comp$-rational variety?
	\end{question}

	\begin{nothing}
		Let $n \geq 2$, $(a_0, \dots, a_n) \in (\Nat^+)^{n+1}$ and $f = X_0^{a_0} + X_1^{a_1} + \dots + X_n^{a_n} \in \Comp[X_0, \dots, X_n]$. Let $L = \lcm(a_0, \dots,a_n)$ and let $\deg(X_i) = w_i = L / a_i$ for each $i \in \{0, \dots, n\}$. Then $f$ is a homogeneous irreducible element of the graded ring $\Comp_{w_0, \dots, n_n}[X_0, \dots,  X_n]$, $B_{a_0,\dots,a_n} = \Comp_{w_0,\dots,w_n}[X_0, \dots,  X_n] / \lb f \rb$ is a graded ring and $\deg(x_i) = w_i$ for each $i \in \{0,\dots, n\}$ where $x_i \in B_{a_0,\dots, a_n}$ is the canonical image of $X_i$.
	\end{nothing}
	
	\begin{proposition}\label{SpecIFFProjRational}
		Let $B = B_{a_0,a_1,a_2,a_3}$. Then $\Spec B$ is rational over $\Comp$ if and only if $\Proj B$ is rational over $\Comp$. 
	\end{proposition}
	\begin{proof}
		Let $K$ be the function field of $\Proj B$ and recall that the function field of $\Spec B$ is isomorphic to $K^{(1)}$. Assume $\Spec B$ is rational over $\Comp$. Then $K^{(1)} \isom \Comp^{(3)}$ and so the function field of $\Proj B$ is stably rational, hence unirational. By Castelnuovo's Theorem, $K$ is rational. The converse is clear.    
	\end{proof}
	
	\begin{definition}\label{cotypeDef}
		Given a tuple $S = (a_0, \dots, a_n) \in (\Nat^+)^{n+1}$, we define $S_i = (a_0, \dots,a_{i-1}, \hat{a}_i, a_{i+1}, \dots, a_n)$ and $L_i = \lcm S_i$. We define $\cotype(S) = |\setspec{i \in \{0, \dots, n\}}{L_i \neq L}|$. 
	\end{definition}
	
	\begin{nothing}\label{reduction}
		Let $B = B_{a_0, \dots, a_n}$ be Pham-Brieskorn ring. Lemma 3.1.14 in \cite{ChitayatThesis} shows that there exists a Pham-Brieskorn ring $B' = B_{b_0, \dots, b_n}$ such that $\cotype(b_0, \dots, b_n) = 0$ and $\Proj B \isom \Proj B'$. Moreover, given a Pham-Brieskorn ring $B$, it is straightforward to determine $B'$ (or equivalently the tuple $(b_0,\dots,b_n)$) explicitly; in particular, the proof of Lemma 3.1.14 in \cite{ChitayatThesis} shows that $b_i \leq a_i$ for all $i \in \{0, \dots, n\}$. This is demonstrated in Example \ref{redExample} below. Consequently, to answer Question \ref{rationalQuestion}, it suffices to consider the special case where $\cotype(a_0, a_1, a_2, a_3) = 0$.    
	\end{nothing}
	
	\begin{proposition}\cite[Proposition 2.4.25]{ChitayatThesis}\label{PBWellFormed}
		Let $f = X_0^{a_0} + \dots + X_n^{a_n}$ where $n \geq 2$ and $a_i \geq 1$ for all $i$. Then, the weighted hypersurface $X_f \subset \PPP = \PPP(w_0, \dots, w_n)$ is a well-formed quasismooth hypersurface if and only if $\cotype(a_0, \dots , a_n) = 0$.
	\end{proposition}
	
	\begin{theorem}\label{PBTHM}
		Suppose $a_0 \leq a_1 \leq a_2 \leq a_3$ and  $\cotype(a_0,a_1,a_2,a_3) = 0$.  Then $\Proj B_{a_0,a_1,a_2,a_3}$ is rational if and only if one of the following holds:
		\begin{enumerate}[\rm(a)]
			\item $a_0 = a_1$, $a_2 = a_3$ and $\gcd(a_0,a_2) = 1$;
			\item $\frac{1}{a_0} + \frac{1}{a_1} + \frac{1}{a_2} + \frac{1}{a_3} > 1$.
		\end{enumerate}
	\end{theorem}
	
	\begin{proof}
		Let $X =  \Proj B_{a_0,a_1,a_2,a_3}$. Let $L' = \lcm(a_0,a_1,a_2,a_3)$ and let $L = \lcm(w_0,w_1,w_2,w_3)$. Since $w_i = L' / a_i$ for each $i \in \{0,1,2,3\}$, we obtain that $L' = nL$ for some $n \in \Nat^+$. Proposition \ref{PBWellFormed} shows that $X$ is a well-formed quasismooth hypersurface of degree $nL$ so $X$ satisfies the hypotheses of Theorem \ref{hypersurfaceDegree}. It is easy to see that condition (a) is equivalent to condition (i) in Theorem \ref{hypersurfaceDegree}. Since  
		$$\frac{1}{a_0} + \frac{1}{a_1} + \frac{1}{a_2} + \frac{1}{a_3} > 1 \iff nL - \sum_{i = 0}^3 w_i < 0, $$ condition (b) is equivalent to condition (ii) in Theorem \ref{hypersurfaceDegree}. These two equivalences prove the theorem. 
		
	\end{proof}
	
	\begin{example}\label{redExample}
		We show that both $\Proj B_{3,7,14,87}$ and $\Spec B_{3,7,14,87}$ are rational. Note that $\cotype(3,7,14,87) = 2$ so we cannot directly apply Theorem \ref{PBTHM}. By Proposition 5.2 in \cite{Chitayat_Daigle_2019}, we find $\Proj B_{3,7,14,87} \isom \Proj B_{3,7,14,3} \isom \Proj B_{3,7,7,3} \isom \Proj B_{3,3,7,7}$. Since $\cotype(3,3,7,7) = 0$, $\Proj B_{3,3,7,7}$ is rational by Theorem \ref{PBTHM}. Consequently, $\Proj B_{3,7,14,87}$ is rational and hence $\Spec B_{3,7,14,87}$ is rational by Proposition \ref{SpecIFFProjRational}.   
	\end{example}
	
	We also obtain the following. Note there is no assumption on the cotype.  
	
	\begin{corollary}\label{ratSings}
		Let $B_{a_0, a_1, a_2, a_3}$ be a Pham-Brieskorn ring and consider the following statements:
		\begin{enumerate}[\rm(a)]
			\item $\Spec B_{a_0,a_1,a_2,a_3}$ has a rational singularity at the origin;
			\item $\frac{1}{a_0} + \frac{1}{a_1} + \frac{1}{a_2} + \frac{1}{a_3} > 1$;       
			\item $\Proj B_{a_0,a_1,a_2,a_3}$ and $\Spec B_{a_0,a_1,a_2,a_3}$ are rational.
		\end{enumerate}
		
		Then $(a) \iff (b) \Rightarrow (c)$. 
		
	\end{corollary}
	\begin{proof}
		
		The equivalence of (a) and (b) is stated in Example 2.21 of \cite{flenner2003rational} so it suffices to prove (b) implies (c). By \ref{reduction}, there exists a tuple $(b_0, b_1, b_2, b_3)$ such that $\Proj B_{a_0,a_1,a_2,a_3} \isom \Proj B_{b_0,b_1,b_2,b_3}$, $\cotype(b_0,b_1,b_2,b_3) = 0$ and $b_i \leq a_i$ for all $i \in \{0,1,2,3\}$. Then $\frac{1}{b_0} + \frac{1}{b_1} + \frac{1}{b_2} + \frac{1}{b_3} \geq \frac{1}{a_0} + \frac{1}{a_1} + \frac{1}{a_2} + \frac{1}{a_3} > 1$. By Theorem \ref{PBTHM} (b), $\Proj B_{b_0,b_1,b_2,b_3}$ is rational and so $\Proj B_{a_0,a_1,a_2,a_3}$ and $\Spec B_{a_0,a_1,a_2,a_3}$ are rational as well.  
	\end{proof}
	
	\section{Proof of Theorem A}\label{Sec:Numerical}
	
	This section provides the proof of Theorem A. Recall the definition of a well-formed tuple from Definition \ref{wellformedtuple}.
	
	\begin{theoremA}\label{TheoremA}
		Suppose $(d_1,d_2,d_3,d_4) \in (\Nat^+)^4$ is well-formed, let $\Gamma = \lb d_1, d_2, d_3, d_4 \rb$ and let $L = \lcm(d_1, d_2, d_3, d_4)$. 
		\begin{enumerate}[\rm(i)]
			\item If $N \geq \max\{ 2L-\sum_{m=1}^4 d_m, 0 \}$ then $N \in \Gamma$. 
			\item  If $n \in \Nat^+$ and $n \cdot L - \sum_{i=1}^4 d_i \in \Nat \setminus \Gamma$, then $n=1$ and $| \{ d_1, d_2 , d_3 , d_4 \} | = 2$.
			
		\end{enumerate}
		
	\end{theoremA} 
	
	The proof of Theorem A is given in  paragraphs \ref{BoundsFrobNbr}--\ref{EndOfProof}. We preserve the following notation until the end of proof. 
	
	\begin{notation}
		Given $d_1, d_2,d_3,d_4 \in \Nat^+$, we define $L = \lcm(d_1,d_2,d_3,d_4)$ and $\Gamma = \lb d_1,d_2,d_3,d_4 \rb$. Given $i,j \in \{1,2,3,4\}$, we define $L_{i,j} = \lcm(d_i,d_j)$ and $g_{i,j}=\gcd(d_i,d_j)$. 
	\end{notation}
	
	\begin{proposition}  \label {BoundsFrobNbr}
		\mbox{ }
		\begin{enumerate}[\rm(a)]
			
			\item \label{hgg2383rynvide} Let $d_1,d_2 \in \Nat^+$ and $n \in \Nat$.
			If $n > L_{1,2} - d_1 - d_2$ and $g_{1,2} \mid n$ then $n \in \langle d_1, d_2 \rangle$.
			
			\item \label {twobounds3}  If $d_1 , d_2 , d_3 \in \Nat^+$ are relatively prime then $F(d_1,d_2,d_3) \le L_{1,2} + L_{1,3} - d_1 -d_2 -d_3$.
			
		\end{enumerate}
	\end{proposition}
	
	\begin{proof}
		The case $g_{1,2}=1$ of (a) is a well-known result of Frobenius; the general case follows.
		For (b), note that  $F(d_1,d_2,d_3) \leq L_{1,2} + g_{1,2}d_3 - d_1 -d_2 -d_3$
		by the $k=3$ case of statement (6) in the introduction of \cite{Brauer1942}.
		Since $\gcd(g_{1,2}, g_{1,3})=1$, we get $g_{1,2} g_{1,3} \mid d_1$, implying that $g_{1,2} d_3$ divides $\frac{d_1d_3}{g_{1,3}} = L_{1,3}$;
		so $g_{1,2} d_3 \le L_{1,3}$ and (b) follows.
	\end{proof}
	
	\begin{assumption}\label {7823yrefj9ronf}
		We assume from this point onward that $d_1,d_2,d_3,d_4 \in \Nat^+$ and $(d_1, d_2, d_3, d_4)$ is well-formed. 
	\end{assumption}
	
	\begin{lemma}  \label {iu7q623ruyehfj}
		If $i,j,k,l$ are such that $\{i,j,k,l\} = \{1,2,3,4\}$, then $g_{k,l} \mid \frac L{L_{i,j}}$.
	\end{lemma}
	
	\begin{proof}
		Let $m = L/L_{i,j}$.
		Since $g_{k,l}$ is relatively prime to $d_i$ and also to $d_j$, it is relatively prime to $L_{i,j}$.
		Since $g_{k,l}$ divides $L = m L_{i,j}$ and is relatively prime to $L_{i,j}$, we obtain $g_{k,l} \mid m$, as desired.
	\end{proof}

	\noindent \textbf{Proof of Theorem A(i).} We may assume that $d_1 \le d_2 \le d_3 \le d_4$.  We may also assume that $d_1 \geq 2$, otherwise the result is trivial. 
	
	Consider first the case where $L_{1,3} = L$.
	Since $d_2$ and $d_4$ divide $L_{1,3}$ and (by Lemma \ref{iu7q623ruyehfj}) $g_{2,4} = 1$, we get $d_2 d_4 \mid L_{1,3} \mid d_1 d_3$,
	so $d_2 d_4 \le d_1 d_3$. Since $d_1 \leq d_2$ and $d_3 \leq d_4$ it follows that $d_1=d_2$ and $d_3=d_4$.
	Note that $d_1 d_3-d_1-d_3 > 0$, because $d_1,d_3\ge2$ and $g_{1,3}= \gcd(d_1,d_2,d_3,d_4)=1$.
	We have $N \geq 2L - \sum_{m=1}^4 d_m = 2(d_1 d_3-d_1-d_3)> d_1 d_3-d_1-d_3$,
	so Proposition \ref{BoundsFrobNbr}(a) gives $N \in \Gamma$, as desired.
	
	Next, consider the case where $L_{1,3} \neq L$.
	Proposition \ref{BoundsFrobNbr}(b) gives $F(d_1,d_3,d_4) \leq L_{1,3} + L_{1,4} - d_1 - d_3 - d_4$, so
	\begin{align*}
		\textstyle
		N - F(d_1,d_3,d_4)
		& \textstyle  \ge 2L - \sum_{m=1}^4 d_m -  L_{1,3} - L_{1,4} + d_1 + d_3 + d_4  \\
		&= (L - L_{1,3}) + (L - L_{1,4}) - d_2 .
	\end{align*}
	We have $(L - L_{1,3}) + (L - L_{1,4}) \ge L - L_{1,3} = L_{1,3} ( \frac L{L_{1,3}} - 1 ) \ge L_{1,3}$, so 
	\begin{equation} \label {8ctxo98b763}
		\textstyle
		N - F( d_1,d_3,d_4) \ge L_{1,3} - d_2.
	\end{equation}
	We claim that  $L_{1,3} - d_2 > 0$.
	Indeed,  $L_{1,3} - d_2 \ge d_3 - d_2 \ge 0$; if  $L_{1,3} - d_2 = 0$ then $L_{1,3} = d_2 = d_3$,
	so $d_1 \mid d_2$ and $d_1 \mid d_3$, so $\gcd(d_1, d_2, d_3) \geq d_1 >1$, a contradiction.
	This shows that $L_{1,3} - d_2 > 0$, so \eqref{8ctxo98b763} gives $N > F(d_1,d_3,d_4) \geq F(\Gamma)$ as required.
	\hfill $\square$

	\begin{assumption} \label {jchfp2o3bew9}
		Let $\alpha = L - \sum_{i=1}^4 d_i$ and assume that $\alpha \in \Nat \setminus \Gamma$.
	\end{assumption}

	\begin{remark}  \label {9823cb9i2390jc}
		In order to finish the proof of Theorem A,
		it suffices to prove that Assumption \ref{jchfp2o3bew9} implies that $| \{ d_1, d_2 , d_3 , d_4 \} | = 2$.
		(This follows from part (i) of Theorem A.)
		Assumption \ref{jchfp2o3bew9} remains in effect until the end of the proof of Theorem A.
	\end{remark}

	\begin{nothing}  \label {hrj5687q32hnn}
		Assumptions \ref{jchfp2o3bew9} immediately imply:
		\begin{enumerate}[\rm(a)]
			
			\item $\alpha > 0$
			
			\item $\min(d_1, d_2, d_3, d_4) \ge 2$
			
			\item $\sum_{i=1}^4 d_i < L$, so in particular $L \notin \{d_1,d_2,d_3,d_4\}$.
			
			\item $| \{ d_1, d_2, d_3, d_4 \} | \ge 2$
			
			\item If $i,j,k \in \{1,2,3,4\}$ are distinct, then $\gcd(g_{i,j}, d_k) = 1$. 
			
		\end{enumerate}
	\end{nothing}

	\begin{lemma}  \label {j68Sh256euj149rilseuay}
		If there exist distinct $i,j \in \{1,2,3,4\}$ such that $d_i \mid d_j$, then $| \{d_1,d_2,d_3,d_4\} | = 2$.
	\end{lemma}
	
	\begin{proof}
		Without loss of generality, we may assume that $d_3 \mid d_4$.
		Then $g_{1,3} = 1 = g_{2,3}$.
		Let $s = d_4/d_3 \in \Nat^+$.
		Define $a = \big\lceil \frac{s+1}{g_{1,2}} \big\rceil$ and $c = a g_{1,2} - (s+1) \in \Nat$.
		We have $\alpha = L - d_1 - d_2 - (s+1)d_3 \equiv -(s+1)d_3 \pmod{g_{1,2}}$ and $c \equiv -(s+1) \pmod{g_{1,2}}$, so $\alpha - c d_3 \equiv 0 \pmod{g_{1,2}}$.
		Since $\alpha \notin \Gamma$ and $c \in \Nat$, we have $\alpha - c d_3 \notin \langle d_1,d_2 \rangle$,
		so Proposition \ref{BoundsFrobNbr}(a) implies that $\alpha - c d_3 \le L_{1,2} - d_1 - d_2$, i.e.,
		\begin{equation} \label {bvcxnc2987q6wt61tfsgdh4}
			0 \le  L_{1,2} - d_1 - d_2 - \alpha + c d_3
			=  L_{1,2} - L + (s+1+c)d_3
			=  L_{1,2} - L + a g_{1,2} d_3 .
		\end{equation}

		Define $s_i = \gcd(s,d_i)$ for $i=1,2$.
		Since $\gcd(s_1,s_2)$ divides $\gcd(d_1,d_2,d_4)=1$,
		and since $\gcd(s_1,d_3)$ divides $\gcd(d_1,d_3,d_4)=1$ and $\gcd(s_2,d_3)$ divides $\gcd(d_2,d_3,d_4)=1$,
		we have 
		\begin{equation}  \label {gmruynigh76r5nqq378b}
			\text{$s_1$, $s_2$ and $d_3$ are pairwise relatively prime.}
		\end{equation}
		We claim that 
		\begin{equation}  \label {kcj3hbhgaklincD}
			\gcd(s,L_{1,2}) = s_1 s_2 .
		\end{equation}
		To see this, first note that $s_i \mid \gcd(s,L_{1,2})$ for each $i=1,2$; so $s_1 s_2 \mid \gcd(s,L_{1,2})$ by \eqref{gmruynigh76r5nqq378b}.
		Conversely, if a prime power $p^i$ divides $\gcd(s,L_{1,2})$ then the fact that $\gcd(s,d_1,d_2)=1$ implies that $p^i$ divides $s_1$ or $s_2$, so $p^i \mid s_1 s_2$,
		showing that $\gcd(s,L_{1,2})$ divides $s_1s_2$.  This proves \eqref{kcj3hbhgaklincD}.
		Define
		$$
		\textstyle     \sigma = \frac{ \lcm(s, L_{1,2}) }{ L_{1,2} } \in \Nat^+
		$$
		and note that 
		$\sigma = \frac{ \lcm(s, L_{1,2}) }{ L_{1,2} } 
		= \frac{ s L_{1,2} }{ L_{1,2} \gcd(s, L_{1,2}) } 
		= \frac{ s }{ \gcd(s,L_{1,2}) } 
		= \frac{ s }{ s_1 s_2 }$ by \eqref{kcj3hbhgaklincD}; so  
		$$
		s = s_1 s_2 \sigma .
		$$
		Since $\gcd(L_{1,2},d_3)=1$, we have
		$L = \lcm( L_{1,2}, sd_3 )
		= \frac{ L_{1,2} \, sd_3 }{\gcd( L_{1,2} , sd_3 )}
		= \frac{ L_{1,2} \, sd_3 }{\gcd( L_{1,2} , s )}
		= \lcm( L_{1,2}, s ) d_3 = L_{1,2} \, \sigma d_3$. So \eqref{bvcxnc2987q6wt61tfsgdh4} gives
		$0 \le  L_{1,2} - L + a g_{1,2} d_3 = L_{1,2} (1 - \sigma d_3) +  a g_{1,2} d_3$ and hence
		\begin{equation}  \label {dj37r83peu9y0lqf}
			L_{1,2} (\sigma d_3 - 1) \le  a  g_{1,2} d_3 .
		\end{equation}
		Since $\gcd(L_{1,2},d_3)=1$ and $d_3>1$, we have $d_3 \nmid L_{1,2}$, so the inequality in \eqref{dj37r83peu9y0lqf} is strict;
		since both sides of this strict inequality are multiples of $g_{1,2}$, it follows that
		$L_{1,2} (\sigma d_3 - 1) \le  a g_{1,2} d_3 - g_{1,2}$, so
		\begin{equation}  \label {uytrvbnm2k3j4hg5f6e7w894j3n4A7n5y}
			\textstyle   \frac{L_{1,2}}{g_{1,2}} \le  \frac{ a d_3 - 1 }{ \sigma d_3 - 1} \, .
		\end{equation}
		
		Consider the case $s=1$. Then $\sigma=1$, so \eqref{uytrvbnm2k3j4hg5f6e7w894j3n4A7n5y} gives
		$\big( \frac{ d_1 }{ g_{1,2} } \big) \big( \frac{ d_2 }{ g_{1,2} } \big) = \frac{L_{1,2}}{g_{1,2}} \le  \frac{ a d_3 - 1 }{ d_3 - 1} \le 2a-1$,
		where the last inequality uses $d_3\ge2$.
		Recall that $a = \big\lceil \frac{s+1}{g_{1,2}} \big\rceil = \big\lceil \frac{2}{g_{1,2}} \big\rceil$.
		If $g_{1,2}=1$ then $a=2$, so $d_1 d_2 \le 3$, which contradicts \ref{hrj5687q32hnn}(b).
		So $g_{1,2}>1$, in which case we have $a=1$, so  $\big( \frac{ d_1 }{ g_{1,2} } \big) \big( \frac{ d_2 }{ g_{1,2} } \big) \le 1$,
		so $d_1=d_2$ and $d_4=sd_3=d_3$, showing that $| \{ d_1, d_2, d_3, d_4 \} | \le 2$.
		In view of \ref{hrj5687q32hnn}(d), this shows that if $s=1$ then  $| \{ d_1, d_2, d_3, d_4 \} | = 2$.
		So, to complete the proof, it suffices to show that
		\begin{equation} \label {tv5nm98ewvhjfxgqyja2xda1}
			\text{the case $s\ge2$ does not occur.}
		\end{equation}
		
		From now-on, we assume that $s\ge2$.
		Since $\gcd(g_{1,2},s)$ divides $\gcd(d_1,d_2,d_4)=1$, we have 
		\begin{equation} \label {bv65T2cbv6vjwu}
			\gcd(g_{1,2},s)=1 .
		\end{equation}
		For each $i \in \{1,2\}$, define $\delta_i = d_i/s_i \in \Nat\setminus\{0\}$.
		For each $i \in \{1,2\}$, $g_{1,2}$ divides $d_i = \delta_i s_i$ and is relatively prime to $s_i$, so $g_{1,2} \mid \delta_i$. Thus,
		\begin{equation}  \label {nvh4hFGbhf7gbiknd3982ZEFWe763grbn9}
			\textstyle
			\frac{ L_{1,2} }{ g_{1,2} }
			= \big( \frac{ d_1 }{ g_{1,2} } \big) \big( \frac{ d_2 }{ g_{1,2} } \big)
			= \big( \frac{ \delta_1 }{ g_{1,2} } \big) \big( \frac{ \delta_2 }{ g_{1,2} } \big) s_1 s_2, \quad
			\text{where $\big(\frac{ \delta_1 }{ g_{1,2} } \big) , \big( \frac{ \delta_2 }{ g_{1,2} } \big) \in \Nat^+$.}
		\end{equation}
		
		We shall now consider three cases ($g_{1,2} = 1$, $g_{1,2} = 2$, $g_{1,2} \ge 3$) and show that none of them is possible.
		Assume that  $g_{1,2} \in \{1,2\}$.
		Note that $g_{1,2} \mid (s+1)$, because \eqref{bv65T2cbv6vjwu} shows that if $g_{1,2}=2$ then $s$ is odd.
		So $a = \big\lceil \frac{s+1}{g_{1,2}} \big\rceil = \frac{s+1}{g_{1,2}}$.
		This together with \eqref{uytrvbnm2k3j4hg5f6e7w894j3n4A7n5y} gives
		$$
		L_{1,2} (\sigma d_3 - 1) \le g_{1,2} a d_3 - g_{1,2} = (s+1)d_3 - g_{1,2} < (s+1)d_3 .
		$$
		In view of \eqref{nvh4hFGbhf7gbiknd3982ZEFWe763grbn9}, it follows that
		\begin{equation}  \label {vbvPncbwerb26b4nm23c56sdf834761}
			\textstyle
			\big( \frac{ \delta_1 }{ g_{1,2} } \big) \big( \frac{ \delta_2 }{ g_{1,2} } \big) 
			= \frac{L_{1,2}}{ s_1 s_2 g_{1,2} }
			<  \frac{ (s+1)d_3 }{ s_1 s_2 g_{1,2} (\sigma d_3 - 1) }
			= \big(\frac1{g_{1,2}}\big) \big(\frac{ s+1}{s} \big) \big( \frac{d_3}{ d_3 - \frac1{\sigma} } \big) \, .
		\end{equation}
		Note that if $\delta_1 = 1 = \delta_2$ then both $d_1 = s_1$ and $d_2=s_2$ divide $d_4$, so in fact $L = d_4$,
		contradicting \ref{hrj5687q32hnn}(c).  So $\delta_1 \delta_2 \ge 2$.
		
		\medskip
		
		If $g_{1,2}=1$ then \eqref{vbvPncbwerb26b4nm23c56sdf834761} gives
		$2 \le \delta_1 \delta_2 < \big(\frac{ s+1}{s} \big) \big( \frac{d_3}{ d_3 - \frac1{\sigma} } \big) \le 3$ (the last inequality because $s,d_3\ge2$), so $\delta_1\delta_2=2$.
		We may assume that $\delta_1=1$ and $\delta_2=2$; then $d_1 = s_1$ and $d_2 = 2s_2$. Since $g_{2,3}=1$, $d_3$ is odd, so $d_3\ge3$.
		If $d_3\neq 3$ then $d_3 \ge5$ and hence $\big(\frac{ s+1}{s} \big) \big( \frac{d_3}{ d_3 - \frac1{\sigma} } \big) < 2$, a contradiction.
		So $(d_1,d_2,d_3) = (s_1,2s_2,3)$. Since $g_{1,2}=1$ and $g_{1,3}=1$, $s_1$ is odd and $s_1 \neq 3$, so $s_1\ge5$ and hence $s\ge5$.
		This implies that  $\big(\frac{ s+1}{s} \big) \big( \frac{d_3}{ d_3 - \frac1{\sigma} } \big) < 2$, a contradiction.
		So the case $g_{1,2}=1$ does not occur.
		
		\medskip
		
		If $g_{1,2}=2$ then \eqref{vbvPncbwerb26b4nm23c56sdf834761} gives $1 \le \big( \frac{ \delta_1 }{ 2 } \big) \big( \frac{ \delta_2 }{ 2 } \big) 
		< \big(\frac1{2}\big) \big(\frac{ s+1}{s} \big) \big( \frac{d_3}{ d_3 - \frac1{\sigma} } \big)$, so
		$\big(\frac{ s+1}{s} \big) \big( \frac{d_3}{ d_3 - \frac1{\sigma} } \big) > 2$.
		Since $1 = \gcd( g_{1,2} , d_4 ) = \gcd(2, d_4)$, $d_4 = sd_3$ is odd, so $d_3$ and $s$ are odd. Hence, $d_3\ge3$ and $s\ge3$.
		Then $2 < \big(\frac{ s+1}{s} \big) \big( \frac{d_3}{ d_3 - \frac1{\sigma} } \big) \le \frac43 \cdot \frac32 = 2$, a contradiction.
		So the case $g_{1,2}=2$ does not occur either. 
		
		\medskip
		
		From now-on, we assume that $g_{1,2} \ge 3$; let us show that this leads to a contradicition.
		The definition $a = \big\lceil \frac{s+1}{g_{1,2}} \big\rceil$ implies that $a < \frac{s+1}{g_{1,2}} + 1$, so $a g_{1,2} < s+1 + g_{1,2}$ and hence
		$a g_{1,2} \le s + g_{1,2}$. 
		So \eqref{uytrvbnm2k3j4hg5f6e7w894j3n4A7n5y} gives
		$L_{1,2} (\sigma d_3 - 1) \le  a g_{1,2} d_3 - g_{1,2} \le  (s + g_{1,2}) d_3 - g_{1,2}$, or equivalently (using \eqref{nvh4hFGbhf7gbiknd3982ZEFWe763grbn9})
		$$
		\textstyle
		\big( \frac{ \delta_1 }{ g_{1,2} } \big) \big( \frac{ \delta_2 }{ g_{1,2} } \big) s_1 s_2 g_{1,2} 
		= L_{1,2} \le \frac {(s + g_{1,2}) d_3 - g_{1,2}} { \sigma d_3-1 }
		= \frac {s d_3} { \sigma d_3-1 } + \frac{ g_{1,2} (d_3 - 1) } { \sigma d_3-1 }
		= \frac {s d_3} { \sigma (d_3 - \frac1{\sigma}) } + \frac{ g_{1,2} (d_3 - 1) } { \sigma  (d_3 - \frac1{\sigma})} \, . 
		$$
		Dividing both sides by $s_1 s_2 g_{1,2}$ and using $s = s_1 s_2 \sigma$ gives
		\begin{equation}  \label {463j5m67g8firtg6hbh567hje238iwsk46rtg} 
			\textstyle
			1
			\le \big( \frac{ \delta_1 }{ g_{1,2} } \big) \big( \frac{ \delta_2 }{ g_{1,2} } \big)
			\le \frac1{ g_{1,2} } \big( \frac {d_3} { d_3 - \frac1{\sigma} } \big) + \frac1s \big( \frac{ d_3 - 1 } { d_3 - \frac1{\sigma}} \big) \, .
		\end{equation}
		If $\sigma \ge2$ then (using $g_{1,2} \ge3$ and $s,d_3 \ge2$)
		$\frac1{ g_{1,2} } \big( \frac {d_3} { d_3 - \frac1{\sigma} } \big) + \frac1s \big( \frac{ d_3 - 1 } { d_3 - \frac1{\sigma}} \big)
		\le \frac13 \cdot \frac2{3/2} + \frac12 \cdot \frac1{3/2} < 1$, which contradicts \eqref{463j5m67g8firtg6hbh567hje238iwsk46rtg}.
		So $\sigma=1$ and $s=s_1s_2$.
		Note that $\gcd(s,d_3) = \gcd(s_1s_2,d_3) = 1$ by \eqref{gmruynigh76r5nqq378b}, and $\gcd( g_{1,2}, s d_3 ) = \gcd( d_1, d_2, d_4 ) =1$; so
		\begin{equation}  \label {cvb765j43kiaxcvhj3h45r6e7}
			\text{$g_{1,2}$, $s$ and $d_3$ are pairwise relatively prime.}
		\end{equation}
		Substituting $\sigma=1$ in \eqref{463j5m67g8firtg6hbh567hje238iwsk46rtg} gives
		\begin{equation} \label {hxcg872352udb}
			\textstyle
			1 \le \big( \frac{ \delta_1 }{ g_{1,2} } \big) \big( \frac{ \delta_2 }{ g_{1,2} } \big)
			\le \frac1{ g_{1,2} } \big( \frac {d_3} { d_3 - 1 } \big) + \frac1s 
			\le \frac2{ g_{1,2} } + \frac1s \, . 
		\end{equation}
		Now  $\frac2{ g_{1,2} } + \frac1s \ge1$ is equivalent to $(g_{1,2} - 2)(s-1) \le 2$, which implies that $(g_{1,2},s)=(3,2)$
		(because $g_{1,2} \ge3$, $s \ge2$ and $\gcd(g_{1,2},s)=1$).
		Then \eqref{hxcg872352udb} gives 
		$1 \le \frac1{ g_{1,2} } \big( \frac {d_3} { d_3 - 1 } \big) + \frac1s = \frac1{ 3 } \big( \frac {d_3} { d_3 - 1 } \big) + \frac12$,
		so $d_3 \in \{2,3\}$, which contradicts \eqref{cvb765j43kiaxcvhj3h45r6e7}.
		So the case $g_{1,2} \ge 3$ does not occur. This proves \eqref{tv5nm98ewvhjfxgqyja2xda1}, and completes the proof of the Lemma.
	\end{proof}

	\begin{remark}  \label {iuytcvbCi0E23m4n567ydh}  
		In order to finish the proof of Theorem A,
		it suffices to prove that there exist distinct $i,j \in \{1,2,3,4\}$ such that $d_i \mid d_j$.
		Indeed, if such $i,j$ exist then Lemma \ref{j68Sh256euj149rilseuay} implies that $| \{d_1,d_2,d_3,d_4\} | = 2$,
		which (by Remark \ref{9823cb9i2390jc}) completes the proof of the Theorem.
		Our proof that $i,j$ exist is by contradiction.
		In other words, we shall prove that Assumptions \ref{8273gbrfo933jd} lead to a contradiction.
	\end{remark}

	\begin{assumptions} \label {8273gbrfo933jd}
		We continue to assume that the conditions of \ref{7823yrefj9ronf} and \ref{jchfp2o3bew9} are satisfied, and we also assume that
		$$
		\text{no elements $i,j \in \{1,2,3,4\}$ satisfy $i \neq j$ and $d_i \mid d_j$.}
		$$
		We will derive a contradiction from the above assumption.
		Without loss of generality, we also assume that the labeling of $d_1,d_2,d_3,d_4$ is such that $\max(d_1,d_2,d_3,d_4) = d_4$.
		Define positive integers $m,f,g$ by
		$$
		\textstyle
		m = \frac L{d_4}, \quad \frac{L}{ d_1 d_2 d_3 } = \frac fg \quad \text{and} \quad \gcd(f,g)=1 .
		$$
		Note that $m\ge2$  (if $m=1$ then $L=d_4$ contradicts \ref{hrj5687q32hnn}(c)).
	\end{assumptions}

	\begin{lemma}  \label {98y772364tyet2hbZxmbv}
		\mbox{ }
		\begin{enumerate}[\rm(a)]
			
			\item $g = g_{1,2} \, g_{1,3} \, g_{2,3}$
			
			\item $g \mid m$
			
		\end{enumerate}
	\end{lemma}
	
	\begin{proof}
		Define $e_p = v_p\big( \frac{L}{ d_1d_2d_3} \big)$ for each prime number $p$, and observe that
		$$
		\textstyle  f = \prod\limits_{e_p > 0} p^{e_p} \quad \text{and} \quad g = \prod\limits_{e_p < 0} p^{-e_p} .
		$$
		If $p$ is a prime factor of $g$ then $v_p(d_1d_2d_3) > v_p(L) \ge \max\limits_{1 \le i \le 3} v_p(d_i)$, so $p$ divides exactly two elements of $\{ d_1, d_2, d_3 \}$
		and consequently there exist distinct $i,j \in \{1,2,3\}$ such that $p \mid g_{i,j}$. Thus,
		\begin{equation}  \label {jchbw98isnsd}
			\text{every prime factor of $g$ is a factor of $g_{1,2} \, g_{1,3} \, g_{2,3}$.}
		\end{equation}
		Conversely, suppose that $p$ is a prime factor of  $g_{1,2} \, g_{1,3} \, g_{2,3}$.
		Then there exists exactly one choice of elements $i<j$ of $\{1,2,3\}$ such that $p \mid g_{i,j}$; we have
		$ v_p(d_1d_2d_3) = v_p( d_i d_j ) = \max( v_p(d_i), v_p(d_j) ) + \min( v_p(d_i), v_p(d_j) ) = v_p(L) + v_p(g_{i,j})$,
		so $e_p = v_p(L) - v_p(d_1d_2d_3) = - v_p( g_{i,j} ) < 0$ and hence $p \mid g$.
		Moreover, $v_p(g) = -e_p = v_p(g_{i,j}) = v_p( g_{1,2} \, g_{1,3} \, g_{2,3} )$. Hence,
		\begin{equation}  \label {en8s3dcyfcjf}
			\text{every prime factor $p$ of $g_{1,2} \, g_{1,3} \, g_{2,3}$ satisfies $v_p(g) = v_p( g_{1,2} \, g_{1,3} \, g_{2,3} )$.}
		\end{equation}
		Statements \eqref{jchbw98isnsd} and \eqref{en8s3dcyfcjf} imply that $g_{1,2} \, g_{1,3} \, g_{2,3} = g$, which proves (a).
		
		Let $i<j$ be elements of $\{1,2,3\}$. Then $\gcd( g_{i,j} , d_4 ) = 1$ and $g_{i,j}$ divides $L = md_4$, so $g_{i,j} \mid m$.
		Since $g_{1,2}, g_{1,3}, g_{2,3}$ are pairwise relatively prime, we obtain $(g_{1,2} \, g_{1,3} \, g_{2,3}) \mid m$, proving (b).
	\end{proof}

	\begin{lemma}  \label {56nOp9zg54yu}
		For every choice of distinct $i,j \in \{1,2,3\}$, we have  $gm \nmid d_i d_j$.
	\end{lemma}
	
	\begin{proof}
		Let $L_{1,2,3} =\lcm(d_1,d_2,d_3)$ and $M = \frac{d_1 d_2 d_3}{ g }$.  We will show that 
		\begin{equation}\label{L123M}
			L_{1,2,3} = M.
		\end{equation} 
		Lemma \ref{98y772364tyet2hbZxmbv}(a) implies that
		$\frac{M}{d_1} = \frac{d_2 d_3}{ g } = \frac{d_2}{ g_{1,2} \, g_{2,3}} \cdot \frac{d_3}{g_{1,3} } \in \Integ$, so $d_1 \mid M$.
		By symmetry, $d_2, d_3 \mid M$, so $L_{1,2,3} \mid M$.
		Conversely, consider a prime factor $p$ of $M$. We can choose $u,v,w$ such that $\{u,v,w\} = \{1,2,3\}$ and $p \nmid d_w$.
		Then $v_p(M) = v_p( \frac{d_u d_v}{ g_{u,v} } ) + v_p( \frac{d_w}{ g_{u,w} \, g_{v,w} } ) = v_p( \frac{d_u d_v}{ g_{u,v} } ) = v_p(L_{u,v}) \le v_p(L_{1,2,3})$,
		and this shows that $M \mid L_{1,2,3}$. This shows \eqref{L123M}.
		
		Arguing by contradiction, suppose that $i,j,k$ are such that $\{i,j,k\} = \{1,2,3\}$ and $gm \mid d_i d_j$.
		Then $\frac{d_1 d_2 d_3}{gm} \in d_k \Integ$. We have $\frac{d_1 d_2 d_3}{g} = M = L_{1,2,3}$ by \eqref{L123M}, so
		$$
		\textstyle
		\frac{d_4}{\big(\frac{d_1 d_2 d_3}{gm}\big)}
		= \frac{m d_4}{\big(\frac{d_1 d_2 d_3}{g}\big)}
		= \frac{L}{L_{1,2,3}} \in \Integ,
		$$
		so $\frac{d_1 d_2 d_3}{gm}$ divides $d_4$. Since $\frac{d_1 d_2 d_3}{gm} \in d_k \Integ$, we have $d_k \mid d_4$, which contradicts \ref{8273gbrfo933jd}.
	\end{proof}

	\begin{lemma}  \label {PrQpo87y3248r3f829jenxcbweu}
		Let $i,j,k$ be such that $\{i,j,k\} = \{1,2,3\}$.
		If $g_{i,j}=1$ then $\frac fg \le \big( \frac{m}{m-1} \big) \big( \frac1{d_i d_j} + \frac1{d_k} \big)$.
	\end{lemma}
	
	\begin{proof}
		Since $g_{i,j} = 1$, Proposition \ref{BoundsFrobNbr}(a) implies
		that every integer strictly greater than $d_i d_j - d_i - d_j$ belongs to $\langle d_i , d_j \rangle$ and hence to $\Gamma$.
		So $\alpha \le d_i d_j - d_i - d_j$, because $\alpha \notin \Gamma$.
		So $m d_4 = L = \alpha + \sum_{\ell=1}^4 d_\ell \le d_i d_j - d_i - d_j + \sum_{\ell=1}^4 d_\ell = d_i d_j + d_k + d_4$ and consequently
		$(m-1) d_4 \le d_i d_j + d_k$. Multiplying both sides by $\frac{m}{(m-1)d_1 d_2 d_3}$  gives
		$$
		\textstyle  \frac{m d_4}{ d_1 d_2 d_3 } \le \frac{m}{m-1} \cdot \frac{d_i d_j + d_k}{ d_1 d_2 d_3 } = \frac{m}{m-1} \big(\frac1{d_k} + \frac1{d_i d_j}) .
		$$
		Since $\frac{m d_4}{ d_1 d_2 d_3 } = \frac{L}{ d_1 d_2 d_3 } = \frac fg$, the result follows.
	\end{proof}

	\begin{lemma}  \label {mrh6t7w3h4cfap}
		We have $g_{i,j} > 1$ for all choices of distinct $i,j \in \{1,2,3\}$.
	\end{lemma}
	
	\begin{proof}
		Proceeding by contradiction, suppose that there exist $i,j,k$ such that  $\{i,j,k\} = \{1,2,3\}$ and $g_{i,j}=1$.
		Define $\delta_i = \frac{d_i}{g_{i,k}}$, $\delta_j = \frac{d_j}{g_{j,k}}$ and $\delta_k = \frac{d_k}{g_{i,k} \, g_{j,k}}$;
		then $\delta_i, \delta_j, \delta_k \in \Nat^+$ and
		\begin{equation} \label {8728379rj93jd} 
			\gcd(\delta_i, \, g_{j,k}\,\delta_k) = 1 = \gcd(\delta_j, \, g_{i,k}\,\delta_k) , \quad
			\gcd(\delta_i,\delta_j)=1, \quad \text{and} \quad\min( \delta_i,\delta_j ) \ge 2 \, .
		\end{equation}
		Indeed,
		$\gcd(\delta_i, \, g_{j,k}\delta_k) = \gcd( \frac{d_i}{g_{i,k}} , \frac{d_k}{g_{i,k}} ) = 1$,
		$\gcd(\delta_j, \, g_{i,k}\delta_k) = \gcd( \frac{d_j}{g_{j,k}} , \frac{d_k}{g_{j,k}} ) = 1$,
		and $\gcd(\delta_i,\delta_j)=1$ because $g_{i,j}=1$;
		we have $\min( \delta_i,\delta_j ) \ge 2$ because $d_i \nmid d_k$ and $d_j \nmid d_k$, so \eqref{8728379rj93jd} is true.
		We have $g = g_{i,k} \, g_{j,k}$ by Lemma \ref{98y772364tyet2hbZxmbv}(a), so
		Lemma \ref{PrQpo87y3248r3f829jenxcbweu} gives $\frac fg  
		\le \big( \frac{m}{m-1} \big) \big( \frac1{d_i d_j} + \frac1{d_k} \big)
		= \big( \frac{m}{m-1} \big) \big( \frac1{\delta_i \delta_j g} + \frac1{\delta_k g} \big)$, i.e.,
		\begin{equation} \label {hjcbut4rt37rh388}
			\textstyle
			1 \le f \le \big( \frac{m}{m-1} \big) \big( \frac1{\delta_i \delta_j} + \frac1{\delta_k} \big) .
		\end{equation}
		If $\delta_k>3$ then $\delta_i \delta_j \ge 6$, so $\frac1{\delta_i \delta_j} + \frac1{\delta_k} < \frac{1}{6} + \frac14 < \frac12$; if $\delta_k=3$ then \eqref{8728379rj93jd} implies that $\delta_i \delta_j \ge 10$, so $\frac1{\delta_i \delta_j} + \frac1{\delta_k} \le \frac{13}{30} < \frac12$;
		since $\frac{m}{m-1} \le 2$, \eqref{hjcbut4rt37rh388} gives $1 \le f < 1$ in both cases, a contradiction. So $\delta_k\in \{1,2\}$.
		
		If $\delta_k=2$ then
		it follows from \eqref{8728379rj93jd} that $\delta_i \delta_j \ge 15$, so \eqref{hjcbut4rt37rh388} gives
		$1 \le f \le \big( \frac{m}{m-1} \big) \frac{17}{30}$, which implies that $m=2$.
		Then $mg=2g = \delta_k g = d_k$ divides $d_id_k$, which contradicts Lemma \ref{56nOp9zg54yu}.
		So $\delta_k=1$. To summarize,
		$$
		(d_i, d_j, d_k) = (\delta_i g_{i,k} , \delta_j g_{j,k}, g) \quad \text{and} \quad  g = g_{i,k} g_{j,k} .
		$$
		Since $d_k \nmid d_i$ and $d_k \nmid d_j$, we have
		\begin{equation}  \label {6547m6n78ehwvqP6543ewsxcv2bhe8rf7d}
			g_{i,k} > 1 \quad \text{and} \quad g_{j,k}>1 . 
		\end{equation}
		
		We have $\delta_i \delta_j \ge6$ by \eqref{8728379rj93jd}, so
		$f \le \big( \frac{m}{m-1} \big) \big( \frac1{\delta_i \delta_j} + \frac1{\delta_k} \big) 
		\le \frac76 \frac{m}{m-1} \le \frac76 \cdot 2 < 3$, so $f \in \{1,2\}$.
		
		If $f=2$ then $2 \le \frac76 \frac{m}{m-1}$ implies that $m=2$, which is impossible because
		$g_{i,k} > 1$, $g_{j,k}>1$ and $g_{i,k} \, g_{j,k} = g$ divides $m$ by Lemma \ref{98y772364tyet2hbZxmbv}(b).  So $f=1$.
		Since $\frac{L}{d_1d_2d_3} = \frac fg$, this means that $L= \frac{d_i d_j d_k}{g} = \frac{d_i d_j g_{i,k} g_{j,k}}{ g }$, so
		$$
		L = d_i d_j = L_{i,j} .
		$$
		It follows that $d_4$ divides $d_id_j = \delta_i \delta_j g_{i,k} g_{j,k}$; since $\gcd(d_4, g_{i,k} g_{j,k}) = 1$,
		we get $d_4 \mid \delta_i \delta_j$ and hence
		\begin{equation}  \label {876tbfgjdke3mejycnx}
			d_4 \le \delta_i \delta_j .
		\end{equation}
		Proposition \ref{BoundsFrobNbr}(b) gives $F( \langle d_1, d_2, d_3 \rangle) \le L_{k,i} + L_{k,j} - d_1 - d_2 - d_3$.
		Note that $L_{k,i} = \frac{d_k d_i}{ g_{i,k} } = \delta_i g$ and $L_{k,j} = \frac{d_k d_j}{ g_{j,k} } = \delta_j g$.
		Since $\alpha \notin \langle d_1, d_2, d_3 \rangle$, we get $\alpha \le  F( \langle d_1, d_2, d_3 \rangle) \le (\delta_i + \delta_j)g - d_1 -d_2-d_3$,
		so $L \le (\delta_i + \delta_j)g + d_4 \le (\delta_i + \delta_j)g + \delta_i \delta_j$ by \eqref{876tbfgjdke3mejycnx}.
		Since $L = d_id_j = \delta_i \delta_j g$, we obtain $\delta_i \delta_j(g-1) \le (\delta_i + \delta_j)g$, so
		$$
		\textstyle    1 - \frac1g \le \frac1{\delta_i} + \frac1{\delta_j} .
		$$
		Condition \eqref{6547m6n78ehwvqP6543ewsxcv2bhe8rf7d} and $\gcd( g_{i,k}, g_{j,k} ) = 1$ imply that $g = g_{i,k} g_{j,k} \ge 6$,
		and \eqref{8728379rj93jd} gives $\frac1{\delta_i} + \frac1{\delta_j} \le \frac12+\frac13$; thus,
		$\frac 56 \le 1 - \frac1g \le \frac1{\delta_i} + \frac1{\delta_j} \le \frac56$.
		This implies that $g=6$ (so $\{g_{i,k}, g_{j,k}\} = \{2,3\}$) and that $\{\delta_i, \delta_j\} = \{2,3\}$.
		We may assume that $\delta_i = 2$ and  $\delta_j = 3$.
		Then \eqref{8728379rj93jd} gives $g_{i,k}=2$ and $g_{j,k}=3$.
		So \eqref{876tbfgjdke3mejycnx} implies that $d_4 \le \delta_i \delta_j = 6 < 9 = d_j$, contradicting $d_4 = \max(d_1,d_2,d_3,d_4)$.
	\end{proof}

	\begin{lemma}  \label {98iuj34nv90es3eikdfn}
		$m \le 6$
	\end{lemma}
	
	\begin{proof}
		Without loss of generality, we may assume that $d_1 \le d_2 \le d_3$; by \ref{8273gbrfo933jd}, it follows that $d_1 < d_2 < d_3 < d_4$.
		Since $\alpha \notin \Gamma$, we have $\alpha \notin \langle d_1, d_2, d_3 \rangle$, so Proposition \ref{BoundsFrobNbr}(b) gives
		$\textstyle L - \sum_{i=1}^4 d_i = \alpha \le F( \langle d_1, d_2, d_3 \rangle ) \le L_{1,2} + L_{1,3} - d_1 - d_2 - d_3$,
		so $(1 - \frac1m)L = L-d_4 \le L_{1,2} + L_{1,3}$ and hence $\frac1{L/L_{1,2}} + \frac1{L/L_{1,3}} \ge (1 - \frac1m)$.
		By contradiction, suppose that $m>6$. Then $\frac1{L/L_{1,2}} + \frac1{L/L_{1,3}} > \frac56$ (where ${L/L_{1,2}}$ and  ${L/L_{1,3}}$ are positive integers), so
		$$
		\textstyle
		\frac L{L_{1,2}} = 1
		\quad \text{or} \quad
		\frac L{L_{1,3}} = 1 
		\quad \text{or} \quad
		\frac L{L_{1,2}} = 2 = \frac L{L_{1,3}} .
		$$
		Also note that $g_{3,4} \mid \frac L{L_{1,2}}$  and $g_{2,4} \mid \frac L{L_{1,3}}$ by Lemma \ref{iu7q623ruyehfj}.
		If $\frac L{L_{1,2}} = 1$ then $d_3$ and $d_4$ divide $L_{1,2}$; since $g_{3,4}=1$ we obtain $d_3 d_4 \mid L_{1,2} \mid d_1 d_2$,
		which is impossible since $d_1 d_2 < d_3 d_4$.  So $\frac L{L_{1,2}} \neq 1$.
		If $\frac L{L_{1,3}} = 1$ then $d_2$ and $d_4$ divide $L_{1,3}$; since $g_{2,4}=1$ we obtain $d_2 d_4 \mid L_{1,3} \mid d_1 d_3$,
		which is impossible since $d_1 d_3 < d_2 d_4$.  So $\frac L{L_{1,3}} \neq 1$.
		It follows that 
		$$
		\textstyle \frac L{L_{1,2}} = 2 = \frac L{L_{1,3}} .
		$$
		If $p$ is a prime factor of $g_{2,4}$ then $p \nmid d_1$ and $p \nmid d_3$, so $p \nmid L_{1,3}$ and $p \mid L_{1,2}$, contradicting $L_{1,2}=L_{1,3}$.
		This shows that $g_{2,4}=1$, and the same argument gives $g_{3,4}=1$. Thus,
		$$
		g_{2,4} = 1 = g_{3,4} .
		$$
		Let $L_{1,2,3} = \lcm(d_1,d_2,d_3)$. Since $L_{1,2}=L_{1,3}$, we have $L_{1,2,3} = L_{1,2} = L/2$,
		so $v_2(L_{1,2,3})+1 = v_2(2L_{1,2,3}) = v_2(L) = \max( v_2(L_{1,2,3}) , v_2(d_4) )$ and hence $v_2(d_4) = v_2(L_{1,2,3})+1$. 
		So $d_4$ is even; since $g_{2,4} = 1 = g_{3,4}$, $d_2$ and $d_3$ are odd.
		Thus, $v_2(d_4) = v_2(L_{1,2,3})+1 = v_2(d_1) + 1$. So we can write
		$$
		d_1 = 2^r \delta_1 \quad \text{and} \quad d_4 = 2^{r+1} \delta_4, \quad \text{where $\delta_1 , \delta_4 \in \Nat^+$ are odd and $r \in \Nat$.}
		$$
		
		If $p$ is a prime factor of $g_{2,3}$ then $p \nmid d_1$, so $v_p(d_2) = v_p(L_{1,2}) = v_p(L_{1,3}) = v_p(d_3)$.
		The fact that  $v_p(d_2) = v_p(d_3)$ for each prime factor $p$ of $g_{2,3}$ implies that 
		$$
		d_2 = a b \quad \text{and} \quad d_3 = ac, \quad \text{where $a,b,c \in \Nat^+$ are odd and pairwise relatively prime.}
		$$
		
		Since $d_2$ is odd and divides $L = 2L_{1,3}$, we have $d_2 \mid L_{1,3}$;
		it follows that $b = \frac{d_2}{a}$ divides $\frac{L_{1,3}}a$, which divides $\frac{d_1 d_3}{a} = d_1 c$; so $b \mid d_1 c$.
		Since $\gcd(b,c)=1$, we get $b \mid d_1$.
		The fact that $\delta_4$ divides $\frac L2 = L_{1,2}$ implies that $\delta_4 \mid d_1 d_2$; since $g_{2,4}=1$, it follows that $\delta_4 \mid d_1$.
		So we have  $b \mid d_1$,  $\delta_4 \mid d_1$ and $\gcd(b,\delta_4)=1$ (because $g_{2,4}=1$); this implies that $b \delta_4 \mid d_1$.
		Since $b \delta_4$ is odd, it follows that $2^r b \delta_4 \mid d_1$ and hence $2^r b \delta_4 \le d_1 < d_4 = 2^{r+1}\delta_4$.
		This implies that $b=1$. So $d_2 = a$ divides $d_3$, a contradiction.
	\end{proof}

	\begin{nothing}  \label {EndOfProof}
		{\bf End of the proof of Theorem A.}
		Lemma \ref{98iuj34nv90es3eikdfn} gives $m \le 6$, and we have $g \mid m$ by Lemma \ref{98y772364tyet2hbZxmbv}(b).
		So $g \mid p^r q^s$ for some prime numbers $p,q$ and some $r,s \in \Nat$.
		Since Lemma \ref{98y772364tyet2hbZxmbv}(a) gives $g = g_{1,2} \, g_{1,3} \, g_{2,3}$ where $g_{1,2}, g_{1,3}, g_{2,3}$ are pairwise relatively prime,
		it follows that one of $g_{1,2}, g_{1,3}, g_{2,3}$  is equal to $1$, contradicting Lemma \ref{mrh6t7w3h4cfap}.
		This shows that Assumptions \ref{8273gbrfo933jd} lead to a contradiction.
		By Remark \ref{iuytcvbCi0E23m4n567ydh}, this completes the proof of Theorem A.
	\end{nothing}
	
	\bibliographystyle{plain}
	\bibliography{RationalityArticleV6}
\end{document}